\documentclass[reqno]{amsart}
\usepackage[utf8]{inputenc}
\usepackage[rgb]{xcolor}
\usepackage[colorlinks, allcolors=blue]{hyperref}
\usepackage{amsmath,amscd,amsthm,amssymb,tikz,ifthen,verbatim,enumerate}
\usetikzlibrary{calc,positioning,arrows}

\allowdisplaybreaks 
\makeatletter \g@addto@macro\@floatboxreset\centering \makeatother 

\usepackage{cleveref} 
\crefname{subsection}{section}{sections}
\Crefname{subsection}{Section}{Sections}

\newtheorem{thm}{Theorem}
\newtheorem{cor}[thm]{Corollary}
\newtheorem{lem}[thm]{Lemma}
\newtheorem{prop}[thm]{Proposition}
\newtheorem{thm*}{Theorem} 
\theoremstyle{definition}
\newtheorem{defn}[thm]{Definition}
\newtheorem{remark}[thm]{Remark}
\newcommand\<{\begin{equation}} \renewcommand\>{\end{equation}}

\newcommand{\C}{{\mathbb C}}
\newcommand{\D}{{\mathbb D}}

\newcommand{\R}{{\mathbb R}}
\newcommand{\Z}{{\mathbb Z}}
\newcommand{\Sb}{{\mathbb S}}

\newcommand{\Bc}{{\mathcal B}}
\newcommand{\Cc}{{\mathcal C}}

\newcommand{\Fc}{{\mathcal F}}
\newcommand{\Gc}{{\mathcal G}}

\newcommand\abs[1]{\left\vert#1\right\vert}

\newcommand\setbuilder[3][:]{\left\{\, #2 #1 #3 \,\right\}}
\let\tilde\widetilde
\renewcommand\bar[1]{\,\overline{\!#1\!}\,}

\providecommand\Id{\mathrm{Id}}

\providecommand\temp{} 

\newcommand\commutativeDiagram[9][1.25]{\raisebox{-3em}{\begin{tikzpicture}[scale=1.5]
    \node (A) at (0,1) {$#2$};
    \node (B) at (#1,1) {$#4$};
    \node (C) at (0,0) {$#7$};
    \node (D) at (#1,0) {$#9$};
	\path[->,font=\scriptsize,>=angle 90]
        (A) edge node [above] {$#3$} (B)
        (A) edge node [left]  {$#5$} (C)
        (B) edge node [right] {$#6$} (D)
        (C) edge node [below] {$#8$} (D);
\end{tikzpicture}}}

\newcounter{commentcounter}

\newcommand\A{{\bar A}}
\renewcommand\P{{\bar P}}
\newcommand\Dual{{\bar D}}
\newcommand\geo{\mathrm{geo}}
\newcommand{\G}{\Gamma} 
\newcommand{\g}{\gamma} 

\DeclareMathOperator{\Cod}{Cod}

\title{Extremal parameters and their duals for Fuchsian boundary~maps}
\date{January 31, 2020}
\author{Adam Abrams}
\makeatletter
\@ifclassloaded{amsart}{
	\address{Institute of Mathematics of the Polish Academy of Sciences, Warsaw, Poland 00656}
	\email{the.adam.abrams@gmail.com}
}{} 
\makeatother

\begin{document}
\begin{abstract}
	We describe arithmetic cross-sections for geodesic flow on compact surfaces of constant negative curvature using generalized Bowen-Series boundary maps and their natural extensions associated to cocompact torsion-free Fuchsian groups. If the boundary map parameters are extremal, that is, each is an endpoint of a geodesic that extends a side of the fundamental polygon, then the natural extension map has a domain with finite rectangular structure, and the associated arithmetic cross-section is parametrized by this set. This construction allows us to represent the geodesic flow as a special flow over a symbolic system of coding sequences. Moreover, each extremal parameter has a corresponding dual parameter choice such that the ``past'' of the arithmetic code of a geodesic is the ``future'' for the code using the dual parameter. This duality was observed for two classical parameter choices by Adler and Flatto; here we show constructively that every extremal parameter set has a dual.
\end{abstract}
\maketitle
\tableofcontents

\section{Introduction} \label{sec intro}

Any closed, oriented, compact surface $M$ of genus $g \ge 2$ and constant negative curvature can be modeled as a quotient $M=\G\backslash\D$, where $\D=\setbuilder{ z \in \C }{ \abs z < 1 }$ is the unit disk endowed with hyperbolic metric
\[ \label{hypmetric} \frac{2 \abs{dz} }{1-{\abs z}^2} \]
and $\G$ is a finitely generated Fuchsian group of the first kind acting freely on $\D$.

Geodesics in this model are half-circles or diameters orthogonal to $\Sb=\partial\D$, the circle at infinity. The geodesic flow $\{\tilde\varphi^t\}$ on $\D$ is defined as an $\R$-action on the unit tangent bundle $S\D$ that moves a tangent vector along the geodesic defined by this vector with unit speed. The geodesic flow $\{\tilde\varphi^t\}$ on $\D$ descends to the geodesic flow $\{\varphi^t\}$ on the factor $M=\Gamma\backslash\D$ via the canonical projection
\< \label{projection} \pi: S\D \to SM \>
of the unit tangent bundles.  The orbits of the geodesic flow $\{\varphi^t\}$ are oriented geodesics on $M$.

A surface $M = \G\backslash\D$ of genus $g$ admits an $(8g-4)$-sided fundamental polygon $\Fc$ (see \Cref{fig irregular sides}) whose sides satisfy the extension condition of Bowen and Series~\cite{BS79}: the geodesic extensions of these segments never intersect the interior of the tiling sets $\gamma \Fc$, $\gamma \in \G$.
We label the sides of $\Fc$ in a counterclockwise order by numbers $1 \le i \le 8g-4$ and label the vertices of $\Fc$ by $V_i$ so that side $i$ connects $V_i$ to $V_{i+1}$, with indices mod $8g-4$. 

\begin{figure}[htb]
	\includegraphics[width=0.5\textwidth]{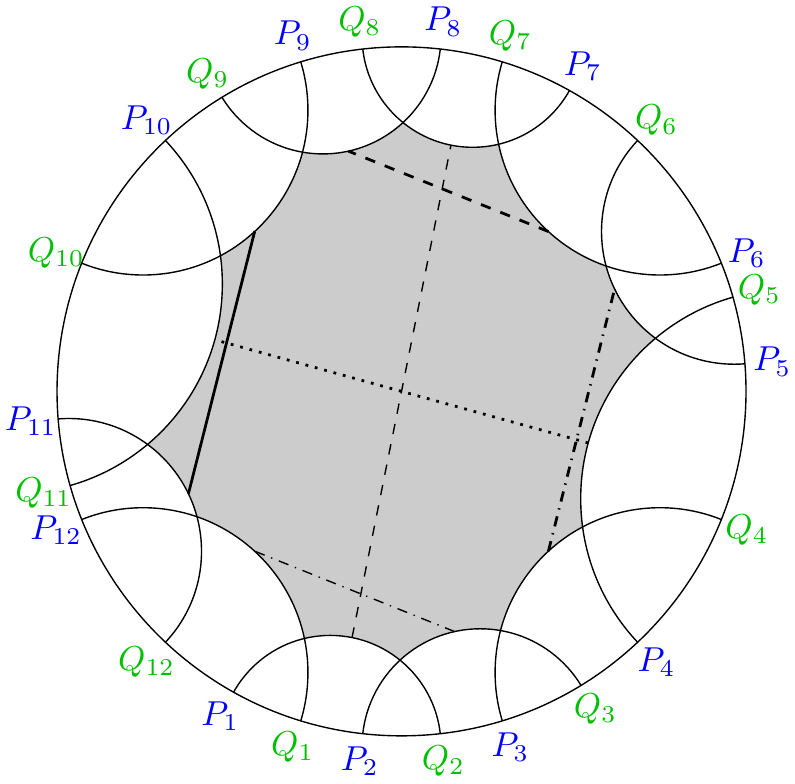}
    \caption{An irregular polygon with side identifications, genus $2$.}
    \label{fig irregular sides}
\end{figure}

We denote by $P_i$ and $Q_{i+1}$ the endpoints of the oriented infinite geodesic that extends side $i$ to the circle at infinity~$\Sb$.\footnote{\,The points $P_i$, $Q_i$ in this paper and~\cite{BS79,KU17,AK19,AKU20} are denoted by $a_i$, $b_{i-1}$, respectively, in~\cite{AF91}.} The order of endpoints on $\Sb$ is the following:
\[ P_1, Q_1, P_2, Q_2, \ldots, P_{8g-4}, Q_{8g-4}. \]
The identification of the sides of $\Fc$ is given by the side pairing rule
\< \label{sigma} \sigma(i) := \left\{ \begin{array}{ll}
    4g-i \bmod (8g-4) & \text{ if $i$ is odd} \\
    2-i \bmod (8g-4) & \text{ if $i$ is even}.
\end{array} \right. \>
The generators $T_i$ of $\G$ associated to this fundamental domain are M\"obius transformations satisfying the following properties, with $\rho(i)=\sigma(i)+1$:
\begin{align*}
    T_i(V_i)&=V_{\rho(i)}\\
    T_{\sigma(i)}T_i&=\Id\\
    T_{\rho^3(i)}T_{\rho^2(i)}T_{\rho(i)}T_i&=\Id.
\end{align*}
According to B.~Weiss~\cite{W92}, the existence of such a fundamental polygon is an old result of Dehn, Fenchel, and Nielsen, while J.~Birman and C.~Series~\cite{BiS87} attribute it to Koebe~\cite{Ko29}. Adler and Flatto~\cite[Appendix A]{AF91} give a careful proof of existence and properties of the fundamental $(8g-4)$-gon for any surface group $\G$ such that $\G\backslash\D$ is a compact surface of genus $g$. Note that in general the polygon $\Fc$ need not be regular. If $\Fc$ is regular, it is the Ford fundamental domain, i.e., $P_iQ_{i+1}$ is the isometric circle for $T_i$, and $T_i(P_iQ_{i+1}) = Q_{\sigma(i)+1}P_{\sigma(i)}$ is the isometric circle for $T_{\sigma(i)}$ so that the inside of the former circle is mapped to the outside of the latter, and all internal angles of $\Fc$ are equal to $\pi/2$.

A \emph{cross-section} $C$ for the geodesic flow $\{\varphi^t\}$ is a subset of the unit tangent bundle $SM$ visited by (almost) every geodesic infinitely often both in the future and in the past. It is well-known that the geodesic flow can be represented as a \emph{special flow} on the space
\[ {C}^{h}=\{\, (v,s) : v\in C, 0\leq s\leq h(v) \,\}. \]
It is given by the formula $\varphi^t(v,s)=(v, s+t)$ with the identification $(v,h(v))=(R(v),0)$, where the ceiling function $h:C\to\R$ is the \emph{time of the first return} of the geodesic defined by $v$ to $C$, and $R:C\to C$ given by $R(v)=\varphi^{h(v)}(v)$ is the \emph{first return map}.

\smallskip
\newcommand\shift{\sigma} 
Let $\mathcal A$ be a finite or countable alphabet, ${\mathcal A}^{\Z}$ be the space of all bi-infinite sequences $\overline x = \{x_i\}_{i \in \Z}$ with topology induced by the metric $d(\overline x, \overline y) = 2^{-\!\min\{\abs i:x_i\ne y_i\}}$. Let $\shift:{\mathcal A}^\Z\to {\mathcal A}^\Z$ be the left shift $(\shift {\overline x})_i=x_{i+1}$, and $\Lambda\subset {\mathcal A}^\Z$ be a closed $\shift$-invariant subset. Then $ (\Lambda,\shift)$ is called a \emph{symbolic dynamical system}. There are some important classes of such dynamical systems. The space $({\mathcal A}^\Z,\shift)$ is called the \emph{full shift}. If the space $\Lambda$ is given by a set of simple transition rules which can be described with the help of a matrix consisting of zeros and ones, we say that $ (\Lambda,\shift)$ is a \emph{subshift of finite type} or a \emph{one-step topological Markov chain} or just a \emph{topological Markov chain}. A factor of a topological Markov chain is called a \emph{sofic shift}. For exact definitions, see~\cite[Section~1.9]{KH}.

\smallskip
In order to represent the geodesic flow as a special flow over a symbolic dynamical system, one needs to choose an appropriate cross-section $C$ and code it, i.e., find an appropriate symbolic dynamical system $ (\Lambda,\shift)$ and a continuous surjective map $\Cod: \Lambda\to C$ (in some cases the actual domain of $\Cod$ is $\Lambda$ except a finite or countable set of excluded sequences) defined such that the diagram
\[ \commutativeDiagram{\Lambda}{\shift}{\Lambda}{\Cod}{\Cod}{C}{R}{C} \]
is commutative. We can then talk about \emph{coding sequences} for geodesics defined up to a shift that corresponds to a return of the geodesic to the cross-section $C$. Notice that usually the coding map is not injective but only finite-to-one (see~\cite[\S{}3.2 and~\S{}5]{A98} and~\cite[Examples~3 and~4]{AK19}).

Slightly paraphrased, Adler and Flatto's method of representing the geodesic flow on $M$ as a special flow is the following. Consider the set of unit tangent vectors $(z,\zeta)\in S\D$ based on the boundary of $\Fc$ and pointed inwards, and denote by $C_\geo$ its image under the canonical projection $\pi:S\D \to SM$.\footnote{\,Adler and Flatto~\cite[p.~240]{AF91} used outward-pointing tangent vectors, but the results are easily transferable.} Every geodesic $\gamma$ in $\D$ is equivalent to one intersecting the fundamental domain $\Fc$ and thus corresponds to a countable set of segments in $\Fc$. More precisely, if we start with a segment $\g\cap\Fc$ which enters $\Fc$ through side $j$ and exits $\Fc$ through side $i$, then $T_i\g\cap\Fc$ is the next segment in $\Fc$, and $T_{j}\g\cap\Fc$ is the previous segment.
The projection $\pi(\g)$ visits $C_\geo$ infinitely often, hence $C_\geo$ is a cross-section, which we call \emph{the geometric cross-section}. The geodesic $\pi(\g)$ can be coded geometrically by a bi-infinite sequence of generators of $\G$ identifying the sides of $\Fc$, as explained below. This method of coding goes back to Morse; for more details see~\cite{K96}.

The set $\Omega_\geo$ of oriented geodesics in $\D$ tangent to the vectors in $C_\geo$ coincides with the set of geodesics in $\D$ intersecting $\Fc$; this is depicted in \Cref{fig rainbow CL} in coordinates $(u,w)\in \Sb\times\Sb, u\neq w$, where $u$ and $w$ are, respectively, the beginning and end of a geodesic. The coordinates of the ``vertices'' of $\Omega_\geo$ are $(P_j,Q_{j+1})$ (the upper part) and $(Q_j,P_j)$ (the lower part). Denote
\< \Gc_i = \setbuilder{ (u,w) }{ \text{geodesic $uw$ exits $\Fc$ through side }i }. \>
Each ``curvilinear horizontal slice'' $\Gc_i$ is contained in the horizontal strip $\Sb \times [P_i,Q_{i+1}]$. The map $F_\geo(u,w)$ piecewise transforms variables by the same M\"obius transformations that identify the sides of $\Fc$, that is,
\<
    F_\geo(u,w)=(T_iu, T_iw) \quad\text{if } (u,w)\in \Gc_i,
\>
and $\Gc_i$ is mapped to the ``curvilineal vertical slice'' belonging to the vertical strip between $P_{\sigma(i)}$ and $Q_{\sigma(i)+1}$. $F_\geo$ is a bijection of $\Omega_\geo$. Due to the symmetry of the fundamental polygon $\Fc$, the ``curvilinear horizontal (vertical) slices'' are congruent to each other by a Euclidean translation. 

For a geodesic $\g$, the \emph{geometric coding sequence}
\[ 
	[\gamma]_\geo = (\ldots,n_{-2},n_{-1},n_0,n_1,n_2,\ldots)
\] 
is such that $\sigma(n_k)$ is the side of $\Fc$ through which the geodesic $F_\geo^k\g$ exits $\Fc$. By construction, the left shift in the space of geometric coding sequences corresponds to the map $F_\geo$, and the geodesic flow $\{\varphi^t\}$ becomes a special flow over a symbolic dynamical system.

\begin{figure}[hbt]
\begin{tikzpicture}
	\draw (0,0) node [left=1em] {\includegraphics[width=0.45\textwidth]{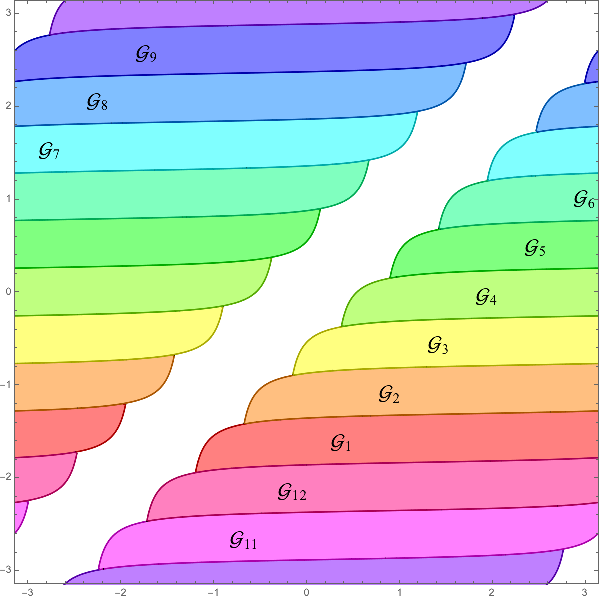}};
	\draw (0,0) node [above] {\;$\stackrel{F_\geo}{\longrightarrow}$};
	\draw (0,0) node [right=1em] {\includegraphics[width=0.45\textwidth]{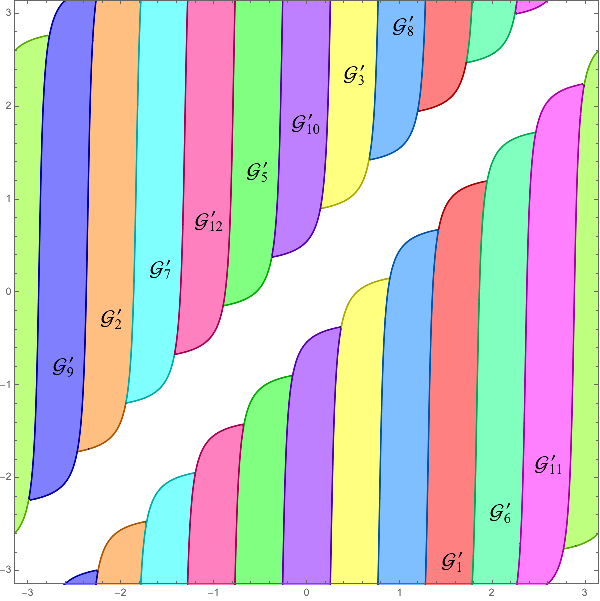}};
\end{tikzpicture}\vspace*{-1em}
\caption{Curvilinear set $\Omega_\geo$ and its image for $g=2$. Here $\Gc'_i = F_\geo(\Gc_i)$.}
\label{fig rainbow CL}
\end{figure}

The set of geometric coding sequences is natural to consider, but it is never Markov. In order to obtain a special flow over a Markov chain, Adler and Flatto replaced the curvilinear boundary of the set $\Omega_\geo$ by polygonal lines, obtaining what they call a ``rectilinear'' set $\Omega_\P$ and defining a map $F_\P$ mapping any horizontal line into a part of a horizontal line. We denote by $F_\geo:\Omega_\geo \to \Omega_\geo$ and $F_\P:\Omega_\P \to \Omega_\P$, respectively, their ``curvilinear'' and ``rectilinear'' transformations. They prove that these maps are conjugate and that the rectilinear map is sofic.

\medskip
In~\cite{KU17,KU17e,AK19,AKU20} the Bowen--Series boundary map $f_\P$ is generalized via a set of $8g-4$ parameters 
\[ \bar A = \{A_1,A_2,\ldots,A_{8g-4}\} \]
with $A_i \in [P_i,Q_i]$.
For any such $\A$ we can define the \emph{boundary map} $f_\A:\Sb \to \Sb$ given by
\< \label{fa} f_\A(x) = T_i(x) \quad\text{if } x \in [A_i,A_{i+1}) \>
and the map $F_\A$ on $\Sb\times\Sb$ given by
\< \label{FA} F_\A(u,w) = (T_i(u),T_i(w)) \quad\text{if } w \in [A_i,A_{i+1}). \>
A priori, using $F_\A$ has little advantage over $f_\A$, but for a suitable domain $\Omega_\A$ we will show that the restriction $F_\A\big|_{\Omega_\A}$ is bijective; the map $F_\A:\Omega_\A \to \Omega_\A$ is called the \emph{natural extension} of the boundary map $f_\A$.

\medskip
There are two important classes of parameters:
\begin{defn}
	We say that $\A$ satisfies the \emph{short cycle property} if~$A_i \in (P_i,Q_i)$ and $f_\A(T_iA_i) = f_\A(T_{i-1}A_i)$ for all $1 \le i \le 8g-4$.
\end{defn}

\begin{defn} 
	\label{defn extremal}
	We say that $\A$ is \emph{extremal} if $A_i \in \{P_i,Q_i\}$ for all $1 \le i \le 8g-4$.
\end{defn}

Adler and Flatto \cite{AF91} dealt exclusively with the cases $\A = \bar P$ and $\A = \bar Q$, which are both extremal, and \cite{AK19} applied the constructions of Adler--Flatto to derive similar results for parameters with short cycles. This paper shows that the relevant results hold for all extremal parameters.

The paper is organized as follows. In \Cref{sec bijectivity}, we show that the map~$F_\A$ with~$\A$ extremal has a bijectivity domain~$\Omega_\A$ with finite rectangular structure. In \Cref{sec conjugacy}, we describe a conjugacy $\Phi$ between $F_\geo:\Omega_\geo\to\Omega_\geo$ and $F_\A:\Omega_\A\to\Omega_\A$ and use this map to define the arithmetic cross-section $C_\A \subset SM$.
In \Cref{sec coding}, the geodesic flow on $M$ is realized as a special flow over a symbolic system $(X_\A,\sigma)$, and we show that for extremal parameters this shift system is always sofic. In \Cref{sec dual}, we prove that every extremal parameter has a ``dual'' parameter choice\footnote{\,By contrast, no instances of duality involving parameters with short cycles are known.} for which the natural extension of the boundary map also has a finite rectangular structure domain; in \Cref{sec examples} we give examples of this duality.

\section{Extremal parameters} \label{sec extremal}

For all of \nameCref{sec extremal}s~\ref{sec bijectivity}-\ref{sec coding} the parameters $\A = \{A_1, A_2, \ldots, A_{8g-4}\}$ will be extremal, that is, for each $i$ either $A_i = P_i$ or $A_i = Q_i$.

\subsection{Bijectivity domains} \label{sec bijectivity}

\begin{figure}[ht]
	\includegraphics[width=0.67\textwidth]{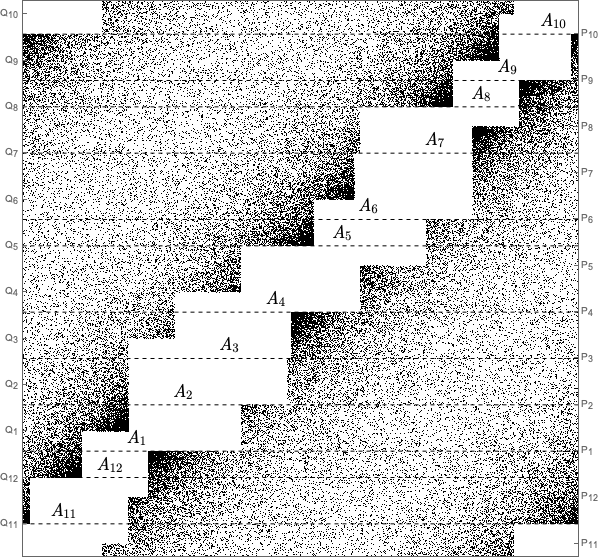}
	\caption{The domain $\Omega_\A$ for the extremal parameters in~\eqref{example A}.}
	\label{fig attractor example}
\end{figure}

This \namecref{sec bijectivity} describes in detail the process by which the domain $\Omega_\A$ of $F_\A$ was realized, culminating with an explicit description (Equation~\ref{Omega E defn}) of $\Omega_A$ in \Cref{thm extremal bijectivity}.

From comparisons of numerical attractors for $F_\A$ with different extremal parameters, we conjectured that the corner points of $\Omega_\A$ for extremal $\A$ should be of the form $(\,\cdot\,,Q_i)$ for lower corners and $(\,\cdot\,,P_i)$ for upper corners.
Recall the definition of $\sigma$ from~\eqref{sigma}, and define
\< \label{tau} \tau(i) = i + (4g-2) \quad\bmod 8g-4. \>
Careful analysis shows that if a set
\providecommand\tempG{x}
\[ \Lambda = \bigcup_{i=1}^{8g-4} [x'_i,\tempG_{i-2}] \times [P_i,Q_i] \cup [x'_i,\tempG_{i-1}] \times [Q_i,P_{i+1}] \]
were to satisfy $F_\A(\Lambda) = \Lambda$ for extremal $\A$, this would imply that 
\[ \left\{ \begin{array}{ll} 
	\tempG_{\sigma(i)} = T_{i} \tempG_{i-2} &\text{if }A_i = P_i \\
	x'_{\tau\sigma(i)} = T_{i-1} x'_i &\text{if }A_i = Q_i.
\end{array} \right. \]
Setting $x'_i = T_{\sigma(i+1)}T_{\tau(i)} \tempG_{\tau(i)}$, these relations can be stated in terms of $\{\tempG_i\}$ only~as
\[ \left\{ \begin{array}{ll} 
	\tempG_{\sigma(i)} = T_{i} \tempG_{i-2} &\text{if }A_i = P_i \\
	\tempG_{\sigma(i)} = T_{\tau(i)+1} \tempG_{\tau(i)} &\text{if }A_i = Q_i.
\end{array} \right. \]
To that end, we prove \Cref{thm system solution} and ultimately \Cref{thm extremal bijectivity}.

\begin{prop} \label{thm system solution}
	For any extremal $\A$, there exist unique values $G_1,\ldots,G_{8g-4}$ such that all of the following hold for all $1 \le i \le 8g-4$:
	\begin{itemize}
		\item $G_i \in [P_i,P_{i+1}]$,
		\item $G_{\sigma(i)} = T_iG_{i-2}$ if $A_i = P_i$,
		\item $G_{\sigma(i)} = T_{\tau(i)+1}G_{\tau(i)}$ if $A_i = Q_i$.
	\end{itemize}
\end{prop}

Before proving \Cref{thm system solution}, it is instructive to work through one example of solving the system
\< \label{system} G_{\sigma(i)} = \left\{ \begin{array}{ll} 
	T_{i} G_{i-2} &\text{if }A_i = P_i \\
	T_{\tau(i)+1} G_{\tau(i)} &\text{if }A_i = Q_i
\end{array} \right. \>
by hand. For this worked example, we consider the genus-$2$ parameter choice
\< \label{example A} \A = \{ P_1, P_2, P_3, P_4, Q_5, P_6, Q_7, Q_8, P_9, P_{10}, Q_{11}, Q_{12} \}. \>
We want to find the values $G_1,\ldots,G_{12}$ described by \Cref{thm system solution}. Applying~\eqref{system} with $i=1,\ldots,8g-4$ gives the following system of twelve equations:
\renewcommand\temp{\text{, we need }}
\begin{enumerate}[\qquad1. {Because\!}]
	\item $A_1 = P_1 \temp G_7 = T_1 G_{11}$.
	\item $A_2 = P_2 \temp G_{12} = T_2 G_{12}$.
	\item $A_3 = P_3 \temp G_5 = T_3 G_1$.
	\item $A_4 = P_4 \temp G_{10} = T_4 G_2$.
	\item $A_5 = Q_5 \temp G_3 = T_{12} G_{11}$.
	\item $A_6 = P_6 \temp G_8 = T_6 G_4$.
	\item $A_7 = Q_7 \temp G_1 = T_{2} G_{1}$.
	\item $A_8 = Q_8 \temp G_6 = T_{3} G_{2}$.
	\item $A_9 = P_9 \temp G_{11} = T_9 G_7$.
	\item $A_{10} = P_{10} \temp G_4 = T_{10} G_8$.
	\item $A_{11} = Q_{11} \temp G_9 = T_{6} G_{5}$.
	\item $A_{12} = Q_{12} \temp G_2 = T_{7} G_{6}$.
\end{enumerate}

Line~1 is not immediately useful since neither $G_7$ nor $G_{11}$ is known. Line~2 tells us that $G_{12}$ is a fixed point of $T_2$. The fixed points of $T_2$ are $P_1$ and $Q_2$, and since we require $G_i \in [P_i,P_{i+1}]$, we must use $G_{12} = P_1$. Likewise, Line~7 gives us that $G_1 = P_1$ as well.
Knowing that $G_1 = P_1$, we can use Line~3 to get $G_5 = T_3P_1$ and then Line~11 to get $G_9 = T_6G_5 = T_6T_3P_1$.

Combing Line~9 with Line~1 tells us that $G_7 = T_1(T_9G_7)$, and so $G_7$ is a fixed point of $T_1T_9$, whose fixed points are $P_8$ and $Q_9$, of which only $G_7 = P_8$ fulfills the requirement $G_i \in [P_i,P_{i+1}]$. Likewise $P_{12}$ is the only valid choice for $G_{11}$ from the two fixed points $P_{12}$ and $Q_1$ of $T_9T_1$. Knowing that $G_{11} = P_{12}$, Line~5 gives us $G_3 = T_{12}G_{11} = T_{12}P_{12} = Q_3$.

The remaining $G_i$ values are determined by the same means. \Cref{fig example graph} shows the system of twelve equations above as a graph, where the $G_i$ are vertices and each equation from the system is an edge. Vertices with a loop are fixed points of some $T_k$, and pairs of vertices in a two-cycle are fixed points of some $T_{\sigma(k)}T_{k+2}$. All other vertices connect through a single path (a ``chain'') to either a fixed point of some $T_k$ or a fixed point of some $T_{\sigma(k)}T_{k+2}$.

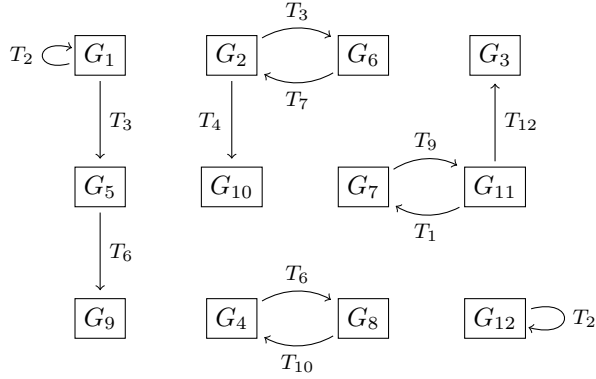
\begin{figure}[htb]
	\begin{tikzpicture}[node distance = 1.75cm]
		\tikzstyle{box}=[draw,rectangle]
		\tikzstyle{arrow}=[->, shorten < = 2pt, shorten > = 3pt, black]
   		\tikzstyle{answer}=[black!33]
    	
		\draw [white] (-2,-4.5) rectangle (7.5,1);
        \node [box] (g1) {$G_1$};
        \node [box,right of = g1] (g2) {$G_2$};
        \node [box,below of = g1] (g5) {$G_5$};
        \node [box,right of = g2] (g6) {$G_6$};
        \node [box,below of = g5] (g9) {$G_9$};
        \node [box,below of = g2] (g10) {$G_{10}$};
        \node [box,right of = g6] (g3) {$G_3$};
        \node [box,below of = g3] (g11) {$G_{11}$};
        \node [box,below of = g10] (g4) {$G_4$};
        \node [box,left of = g11] (g7) {$G_7$};
        \node [box,right of = g4] (g8) {$G_8$};
        \node [box,right of = g8] (g12) {$G_{12}$};
        
        \let\sz\footnotesize
        \path [arrow] (g1) edge node [right] {\sz$T_{3}$} (g5);
        \path [arrow] (g5) edge node [right] {\sz$T_{6}$} (g9);
        \path [arrow] (g2) edge node [left] {\sz$T_{4}$} (g10);
        \path [arrow] (g11) edge node [right] {\sz$T_{12}$} (g3);
        \path [arrow] (g1) edge [loop left] node {\sz$T_{2}$} (g1);
        \path [arrow] (g12) edge [loop right] node {\sz$T_{2}$} (g12);
        \path [arrow] (g2) edge [bend left] node [above] {\sz$T_{3}$} (g6);
        \path [arrow] (g6) edge [bend left] node [below] {\sz$T_{7}$} (g2);
        \path [arrow] (g7) edge [bend left] node [above] {\sz$T_{9}$} (g11);
        \path [arrow] (g11) edge [bend left] node [below] {\sz$T_{1}$} (g7);
        \path [arrow] (g4) edge [bend left] node [above] {\sz$T_{6}$} (g8);
        \path [arrow] (g8) edge [bend left] node [below] {\sz$T_{10}$} (g4);
        
	\end{tikzpicture}
	\caption{The graph representing the system of equations~\eqref{system} for the example parameters~$\A$ in~\eqref{example A}.}
	\label{fig example graph}
\end{figure}

After resolving all loops, two-cycles, and chains in \Cref{fig example graph}, we get that the values $G_i$ from \Cref{thm system solution} for the specific example $\A$ in~\eqref{example A} are
\begin{align}
	G_1&=P_1  &  G_4&=P_5     &  G_7&=P_8        &  G_{10}&=T_4P_2 \nonumber \\*
	G_2&=P_2  &  G_5&=T_3P_1  &  G_8&=P_9        &  G_{11}&=P_{12} \label{example G} \\*
	G_3&=Q_3  &  G_6&=P_6     &  G_9&=T_6T_3P_1  &  G_{12}&=P_1. \nonumber
\end{align}

\label{good paragraph}The proof of \Cref{thm system solution} presented in the remainder of this \namecref{sec bijectivity} shows that the graphs for any extremal $\A$ in any genus always a structure similar to the graph in \Cref{fig example graph}: there are loops and two-cycles that lead directly to values for certain $G_i$, and there are chains connecting unambiguously to a cycle, which determine the remaining $G_i$. The ``types'' in \Cref{def types} provide a rigorous way to describe whether a particular $G_i$ is part of a cycle (Types~\ref{type-P-loop} and~\ref{type-Q-loop}) or part of a chain (Types~\ref{type-PQ} and~\ref{type-QP}), and \Cref{lem infinite +4g} can be interpreted to mean that all chains must end (in either a loop or a two-cycle).

\begin{samepage}
\begin{defn} \label{def types} We say that an index $i$ is
\begin{enumerate}[\quad\text{-- ``Type }1\text{'' if}]
	\item \label{type-P-loop} $A_{\sigma(i)} = P_{\sigma(i)}$ and $A_{i+2} = P_{i+2}$,
	\item \label{type-PQ} $A_{\sigma(i)} = P_{\sigma(i)}$ and $A_{i+2} = Q_{i+2}$,
	\item \label{type-Q-loop} $A_{\sigma(i)} = Q_{\sigma(i)}$ and $A_{\tau(i)} = Q_{\tau(i)}$,
	\item \label{type-QP} $A_{\sigma(i)} = Q_{\sigma(i)}$ and $A_{\tau(i)} = P_{\tau(i)}$.
\end{enumerate}
Note that each index is of exactly one type.
\end{defn}
\end{samepage}

\begin{lem} \label{lem infinite +4g}
	It cannot be that $i + 4gn$ is Type~\ref{type-PQ} for all $n \ge 0$. It cannot be that $i - 4gn$ is Type~\ref{type-QP} for all $n \ge 0$.
\end{lem}

\begin{proof}
Assume that $i + 4gn$ is Type~\ref{type-PQ} for all $n \ge 0$. That is, $A_{\sigma(i + 4gn)} = P_{\sigma(i + 4gn)}$ and $A_{(i + 4gn) + 2} = Q_{(i + 4gn) + 2}$ for all $n \ge 0$. Since $\P$ and $\bar Q$ are disjoint, this requires $\sigma(i + 4gn) \not\equiv i + 4gn + 2~({\rm mod}~{8g-4})$ for all $n \ge 0$. 
By direct calculation, 
\[ \sigma(j) \equiv j+2 \quad\iff\quad j \in \{2g-1,4g-2,6g-3,8g-4\}. \]
Since we are working mod $8g-4$, we can say that $\sigma(j) \equiv j+2$ whenever $j$ is \textit{any} integer multiple of $2g-1$. Thus $i+4gn$ is not a multiple of $2g-1$ for any $n$.

\medskip
If $i$ is even, let $n = K(2 g - 1) - \tfrac i2$ with $K$ large enough that $n > 0$. Then
\begin{align*}
    i + 4gn 
    &= i + 4g \big( 2Kg - K - \tfrac i2 \big) \\
    &= i + 8Kg^2 - 4Kg - 2gi \\
    &= (2g-1)(4 K g - i).
\end{align*}
If $i$ is odd, let $n = K(2g-1) + (g - 1) - \tfrac{i - 1}2$ with $K$ large enough that $n > 0$. Then
\begin{align*}
    i + 4gn 
    &= i + 4g \big( 2Kg-K + g - 1 - \tfrac {i-1}2 \big) \\
    &= i + 8Kg^2 - 4Kg + 4g^2 - 2g - 2gi \\
    &= (2g-1)(4Kg + 2g - i).
\end{align*}
In either case, there exists $n>0$ such that $i+4gn$ is a multiple of $2g-1$.
As shown above, the assumption that $i + 4gn$ is Type~\ref{type-PQ} for all $n \ge 0$ implies that $i+4gn$ is not a multiple of $2g-1$ for any $n$. This contradiction shows that $i+4gn$ cannot be Type~\ref{type-PQ} for all $n\ge0$.

Similar arguments give a contradiction if $i-4gn$ is Type~\ref{type-QP} for all $n \ge 0$.
\end{proof}

\begin{lem} \label{lem range}
	If $x \in [P_j,P_{j+1}]$ then $T_{\tau\sigma(j)+2} T_{j+1} x \in [P_{\tau(j)-2}, P_{\tau(j)-1}]$.
\end{lem}

\begin{proof}
Firstly,
\[ T_{j+1}[P_j,P_{j+1}] = [T_{j+1}P_j, T_{j+1}P_{j+1}] = [P_{\tau\sigma(j)}, Q_{\tau\sigma(j)}]. \]
Setting $k = \tau\sigma(j)$, we now want to show that $T_{k+2}[P_k,Q_k] \subset [P_{\sigma(k)-2}, P_{\sigma(k)-1}]$.
For any $k$, we have ``$[P_k,Q_k] \subset [Q_{k+3},P_{k+1}]$'' in the sense that moving counter-clockwise around the circle one has $Q_{k+3},P_k,Q_k,P_{k+1}$ in that order. Applying $T_{k+2}$, which is monotonically ``increasing'' (counter-clockwise) on $\Sb$, we get
\[ T_{k+2}[P_k,Q_k] \subset [T_{k+2}Q_{k+3},T_{k+2}P_{k+1}] =[P_{\sigma(k)-2},P_{\sigma(k)-1}] \]
as desired.
\end{proof}

\begin{proof}[Proof of \Cref{thm system solution}]
We will construct the solution to the system~\eqref{system}, that is,
\[ G_{\sigma(i)} = \left\{\begin{array}{ll} T_iG_{i-2} &\text{if }A_i=P_i \\ T_{\tau(i)+1}G_{\tau(i)} &\text{if }A_i=Q_i, \end{array}\right. \]
with $G_i \in [P_i,P_{i+1}]$.

If $i$ is Type~\ref{type-P-loop}, then~\eqref{system} tells us that
\[ T_{\sigma(i)}G_{\sigma(i)-2} = G_{i} \]
because $A_{\sigma(i)} = P_{\sigma(i)}$, and because $A_{i+2} = P_{i+2}$  we have
\[ T_{i+2}G_{i} = G_{\sigma(i)-2}. \]
Together, these imply that $G_i$ is a fixed point of $T_{\sigma(i)}T_{i+2}$. The fixed points of that map are $P_{i+1}$ and $Q_{i+2}$, and since we require $G_i \in [P_i,P_{i+1}]$ it must be that $G_i = P_{i+1}$.

\medskip
If $i$ is Type~\ref{type-Q-loop}, then
\begin{align*}
	T_{\tau\sigma(i)+1}G_{\tau\sigma(i)} &= G_{i} \\*
	T_{i+1}G_{i} &= G_{\tau\sigma(i)},
\end{align*}
and so $G_i$ must be the fixed point $P_i$ of $T_{\tau\sigma(i)+1}T_{i+1}$.

\medskip
If $i$ is Type~\ref{type-PQ}, the condition $A_{\sigma(i)} = P_{\sigma(i)}$ implies
\[
    G_{i} = T_{\sigma(i)}G_{\sigma(i)-2},
\]
and the condition $A_{i+2} = Q_{i+2}$ implies
\begin{align*}
    G_{\sigma(i+2)} &= T_{\tau(i+2)+1}G_{\tau(i+2)} \\*
    G_{\sigma(i)-2} &= T_{\tau(i)+3}G_{\tau(i)+2}.
\end{align*}
Using that $\tau(i) = i + 4g-2$, we have 
\< \label{i is Type 2}
	i\text{ is Type~\ref{type-PQ}} \quad\implies\quad
	G_{i} = T_{\sigma(i)} T_{\tau(i)+3}G_{i + 4g}.
\>
Consider the possible types of $i+4g$.
\begin{itemize}
	\item If $i+4g$ is Type~\ref{type-P-loop}, then $G_{i+4g} = P_{i+4g+1}$, and by \Cref{lem range} we \\ get $G_i \in [P_i,P_{i+1}]$.
	\item If $i+4g$ is Type~\ref{type-Q-loop}, then $G_{i+4g} = P_{i+4g}$, and by \Cref{lem range} we \\ get $G_i \in [P_i,P_{i+1}]$. 
	\item If $i+4g$ is Type~\ref{type-QP}, then $A_{\tau(i+4g)} = A_{i+2} = P_{i+2}$, which contradicts $i$ being Type~\ref{type-PQ}. So this cannot occur.
	\item If $i+4g$ is Type~\ref{type-PQ}, the condition $A_{\sigma(i+4g)} = A_{\tau\sigma(i)-2} = P_{\tau\sigma(i)-2}$ gives
\[ 
	G_{i+4g} = T_{\tau\sigma(i)-2} G_{\tau\sigma(i)-4}
\]
and the condition $A_{\tau(i)+4} = Q_{\tau(i)+4}$ gives
\[
	G_{\sigma\tau(i)-4} = T_{i+5} G_{i+4} = T_{i+5} G_{i+8g}.
\]
Combining these with~\eqref{i is Type 2}, we get 
\< \label{i and i+4g are Type 2}
	G_{i} = T_{\sigma(i)} T_{\tau(i)+3} T_{\tau\sigma(i)-2} T_{i+5} (G_{i+8g}).
\>
\end{itemize}

Notice that~\eqref{i is Type 2} gives $G_i$ in terms of $G_{i+4g}$, and now assuming $i+4g$ is Type~\ref{type-PQ} as well we have $G_i$ in terms of $G_{i+8g}$ from~\eqref{i and i+4g are Type 2}. Inductively, we continue to get expressions for $G_i$ in terms of $G_{i+4gn}$ so long as $G_{i+4g(n-1)}$ is Type~\ref{type-PQ}. As with the third bullet above, any $i+4gn$ being Type~\ref{type-QP} will cause an immediate contradiction, and by \Cref{lem infinite +4g} we cannot have $i+4gn$ be Type~\ref{type-PQ} for all $n \ge 0$. Thus for some $n$ we must have $i+4gn$ of Type~\ref{type-P-loop} or~\ref{type-Q-loop}, at which point $G_{i+4gn}$ will be $P_{i+4gn+1}$ or $P_{i+4gn}$, respectively, and inductively this leads to a value for $G_i$ that by \Cref{lem range} will be in $[P_i,P_{i+1}]$.

\medskip
The case where $i$ is Type~\ref{type-QP} is handled similarly to Type~\ref{type-PQ}: we cannot have $i, i-4g, i-8g, i-12g$, etc., all be Type~\ref{type-QP}, none of them can be Type~\ref{type-PQ} without contradiction, and once we have $i-4gn$ being Type~\ref{type-P-loop} or Type~\ref{type-Q-loop} we inductively get a unique value for $G_{i-4gk}$ in $[P_{i-4gk},P_{i-4gk+1}]$ for $k = n,n-1,\ldots,0$.
\end{proof}

\begin{lem} \label{lem DGH properties}
	Let $G_1, G_2, \ldots, G_{8g-4}$ be from \Cref{thm system solution}, and denote 
	\< \label{defn H D}
		H_i := U_i G_{\tau(i)-1}
		\qquad\text{and}\qquad
		D_i := T_{\tau\sigma(i)+1} G_{\tau\sigma(i)},
	\>
	where
	\[ \label{defn U}
		U_i := T_{\sigma(i-1)}T_{\tau(i)} = T_{\sigma(i)}T_{\tau(i)-1}.
	\]
	Then the following hold:
	\begin{enumerate}[\quad 1)]
		\item $D_i = T_{\sigma(i)}H_{\sigma(i)+1}$,
		\item $D_i \in [P_i,Q_i]$,
		\item $H_i \in [Q_i,Q_{i+1}]$.
	\end{enumerate}
\end{lem}

\begin{proof}
	The first item follows directly from the definitions of $H_k$, $U_k$, and $D_i$:
	\begin{align*}
    	T_{\sigma(i)} H_{\sigma(i)+1} 
		&= T_{\sigma(i)} U_{\sigma(i)+1} G_{\tau(\sigma(i)+1)-1} 
		= T_{\sigma(i)} T_{i} T_{\tau\sigma(i)+1} G_{\tau\sigma(i)} \\*
		&= T_{\tau\sigma(i)+1} G_{\tau\sigma(i)} 
		= D_i.
	\end{align*}
	The latter items follow from simple calculations and the fact that $G_k \in [P_k,P_{k+1}]$ for all $k$:
	\begin{align*}
		D_i &\in T_{\tau\sigma(i)+1}[P_{\tau\sigma(i)}, P_{\tau\sigma(i)+1}] = [P_i, Q_i], \\
		H_i &\in U_i[P_{\tau(i)-1},P_{\tau(i)}] = [T_{\sigma(i)}T_{\tau(i)-1}P_{\tau(i)-1}, T_{\sigma(i)}T_{\tau(i)-1}P_{\tau(i)}] \\*&\qquad\qquad= [T_{\sigma(i)} Q_{\sigma(i)+2}, T_{\sigma(i)} P_{\sigma(i)}] = [Q_i,Q_{i+1}].
		\qedhere
	\end{align*}
\end{proof}

\medskip
We are now ready to prove one of the main results of this paper:
\begin{thm} \label{thm extremal bijectivity}
	For any extremal parameter choice $\A$, the set
	\< \label{Omega E defn} \Omega_\A := \bigcup_{i=1}^{8g-4} \; [H_{i+1},G_{i-2}] \!\times\! [P_i,Q_i] \;\cup\; [H_{i+1},G_{i-1}] \!\times\! [Q_i,P_{i+1}] \>
	is a bijectivity domain for $F_\A$, where $\{G_i\}$ are from \Cref{thm system solution} and $\{H_i\}$ are from Equation~\ref{defn H D}.
\end{thm}
\providecommand\FEL{F_\A(\Omega_\A)}

\begin{proof} 
Label the horizontal strips of $\Omega_\A$ as
\< \label{Omega E horizontal strips} \begin{split}
    R_i'' &:= [H_{i+1},G_{i-2}] \times [P_i,Q_i] \\
    R_i' &:= [H_{i+1},G_{i-1}] \times [Q_i,P_{i+1}],
\end{split} \>
so $\Omega_\A = \bigcup_{i=1}^{8g-4} (R_i'' \cup R_i')$.
From \Cref{lem DGH properties}(1), we have that $T_iH_{i+1} = D_{\sigma(i)}$, where $\{D_i\}$ are defined in~\eqref{defn H D}, so we compute that
\begin{align*}
    T_i R_i' 
    &= [T_iH_{i+1},T_iG_{i-1}] \times [T_iQ_i,T_iP_{i+1}] \\*
    &= [D_{\sigma(i)},D_{\sigma(i)+1}] \times [Q_{\sigma(i)+2},P_{\sigma(i)-1}].
\end{align*}
We also compute
\begin{align*}
    T_i R_i'' 
    &= [T_iH_{i+1}, T_iG_{i-2}] \times [T_iP_i, T_iQ_i] \\*
    &= [D_{\sigma(i)},T_iG_{i-2}] \times [Q_{\sigma(i)+1},Q_{\sigma(i)+2}] \\*
    &= [D_{\sigma(i)},G_{\sigma(i)}] \times [Q_{\sigma(i)+1},Q_{\sigma(i)+2}] \qquad\text{if }A_i=P_i
\end{align*} and \begin{align*}
    T_{i-1} R_i'' 
    &= [T_{i-1}H_{i+1}, T_{i-1}G_{i-2}] \times [T_{i-1}P_i, T_{i-1}Q_i] \\*
    &= [T_{i-1}H_{i+1},D_{\tau\sigma(i)+2}] \times [P_{\tau\sigma(i)},P_{\tau\sigma(i)+1}] \\*
    &= [H_{\tau\sigma(i)},D_{\tau\sigma(i)+2}] \times [P_{\tau\sigma(i)},P_{\tau\sigma(i)+1}] \qquad\text{if }A_i=Q_i.
\end{align*}
The image of $\Omega_\A$ under $F_\A$ is therefore
\begin{align*}
    \hspace*{-2em} \FEL
    &= \bigcup_{i=1}^{8g-4} T_\ell R_i'' \cup T_i R'_i, \quad\text{where } \ell = \big\{ \begin{smallmatrix} i\phantom{-1} &\text{if } A_i = P_i \\ i-1 &\text{if } A_i = Q_i \end{smallmatrix} \\
    &= \bigcup_{i=1}^{8g-4} T_iR_i' \;\;\cup \bigcup_{i:\: A_i = P_i}\!\! T_iR_i'' \;\;\cup \bigcup_{i:\: A_i = Q_i}\!\! T_{i-1}R_i'' \\
	&= \bigcup_{i=1}^{8g-4} [D_{\sigma(i)},D_{\sigma(i)+1}] \times [Q_{\sigma(i)+2},P_{\sigma(i)-1}] \\*
	&\qquad \cup \bigcup_{i:\: A_i = P_i}\!\! [D_{\sigma(i)},G_{\sigma(i)}] \times [Q_{\sigma(i)+1},Q_{\sigma(i)+2}] \\*
	&\qquad \cup \bigcup_{i:\: A_i = Q_i}\!\! [H_{\tau\sigma(i)},D_{\tau\sigma(i)+2}] \times [P_{\tau\sigma(i)},P_{\tau\sigma(i)+1}] \\
	&= \bigcup_{j=1}^{8g-4} [D_j,D_{j+1}] \times [Q_{j+2},P_{j-1}] \\*
	&\qquad \cup \bigcup_{i:\: A_i = P_i} [D_j,G_j] \times [Q_{j+1},Q_{j+2}], \qquad j=\sigma(i) \\*
	&\qquad \cup \bigcup_{i:\: A_i = Q_i} [H_k,D_{k+2}] \times [P_k,P_{k+1}], \qquad k=\tau\sigma(i).
\end{align*}

Note that the images $T_iR_i'$ are quite tall since $[Q_{j+2},P_{j-1}]$ contains more than half of $\Sb$. The images $T_\ell R_i''$ are not tall (regardless of whether $A_i$ is $P_i$ or $Q_i$).

Using the inclusions of \Cref{lem DGH properties}, we have
\< \label{image with conditions} \begin{split}
	\FEL 
	&= \bigcup_{j=1}^{8g-4} \underbrace{[D_j,D_{j+1}]}_{\subset\; [P_j,Q_{j+1}]} \times [Q_{j+2},P_{j-1}] \\*
	&\qquad \cup \bigcup_{i:\: A_i = P_i} \underbrace{[D_j,G_j]}_{\subset\; [P_j,P_{j+1}]} \times [Q_{j+1},Q_{j+2}], \qquad j=\sigma(i) \\*
	&\qquad \cup \bigcup_{i:\: A_i = Q_i} \underbrace{[H_k,D_{k+1}]}_{\subset\; [Q_k,Q_{k+1}]} \times [P_{k-1},P_k], \qquad k=\tau\sigma(i)+1.
\end{split} \>

\providecommand\BIGCUP{ \hspace{-0.8em}\bigcup_{\substack{j \ne m\!-\!2, j \ne m\!-\!1, \\ j \ne m,\, j \ne m\!+\!1}}\hspace{-1.333em} }

The right-hand side of~\eqref{image with conditions} has expressions for $3\times(8g-4)$ rectangles.\footnote{\,For any particular $\A$, only $2\times(8g-4)$ of the rectangles in~\eqref{image with conditions} will be part of $\FEL$ since each $i$ satisfies only one of $A_i=P_i$ or $A_i=Q_i$.} Which of these intersect $\Sb \times [P_m,Q_m]$ for a fixed $m$? Looking at the $y$-projection of these rectangles, we see that
\begin{itemize}
	\item $[Q_{j+2},P_{j-1})$ intersects $[P_m,Q_m]$ if and only if $j \notin \{m\!-\!2,\,m\!-\!1,\,m,\,m\!+\!1\}$;
	\item $(Q_{j+1},Q_{j+2}]$ intersects $[P_m,Q_m]$ exactly when $j=m-2$; 
	\item $[P_{k-1},P_k)$ intersects $[P_m,Q_m]$ exactly when $k=m+1$. 
\end{itemize}	
This gives us a smaller collection of rectangles whose union is contained in $\FEL$ and entirely contains the intersection of $\FEL$ with $\Sb \times [P_m,Q_m]$.
\begin{align*}
    \FEL \cap \big(\Sb \times [Q_m,P_{m+1}]\big)
    \subset \left(\begin{array}{l}
    	[D_{m-2},G_{m-2}] \times [Q_{m-1},Q_m] \\
    	\cup\; [H_{m+1},D_{m+2}] \times [P_m,P_{m+1}] \\
    	\displaystyle \cup \BIGCUP [D_j,D_{j+1}] \times [Q_{j+2},P_{j-1}]
	\end{array}\!\!\!\right)
	\subset \FEL.
\end{align*}
The rectangles on the right-hand side above can each be vertically truncated to contain only their intersection with $\Sb \times [P_m,Q_m]$, giving that
\begin{align*}
    &\hspace*{-2em} \FEL \cap \big(\Sb \times [Q_m,P_{m+1}]\big) \\
    &= [D_{m-2},G_{m-2}] \times [P_m,Q_m] 
    \cup [H_{m+1},D_{m+2}] \times [P_m,Q_m] \\*&\qquad
    \cup \BIGCUP [D_j,D_{j+1}] \times [P_m,Q_m] \\
    &= \left([D_{m-2},G_{m-2}]  
    \cup [H_{m+1},D_{m+2}] 
    \cup \BIGCUP [D_j,D_{j+1}] 
    \right) \times [P_m,Q_m] \\
    &= \left(
    [H_{m+1},D_{m+2}] 
    \cup \left(\begin{array}{l} [D_{m+2},D_{m+3}] \\ \cup [D_{m+3},D_{m+4}] \\ \cup \cdots \\ \cup [D_{m-3},D_{m-2}] \end{array}\right)
    \cup [D_{m-2},G_{m-2}]
    \right) \times [P_m,Q_m] \\
    &= [H_{m+1},G_{m-2}] \times [P_m,Q_m]
\end{align*}
This is exactly $R_m''$. Similar arguments show that $\FEL \cap \big(\Sb \times [Q_m,P_{m+1}]\big)$ is exactly $R_m'$, and therefore $\FEL = \bigcup_{m=1}^{8g-4} (R_m'' \cup R_m') = \Omega_\A$.
\end{proof}

\subsection{Conjugacy and cross-section} \label{sec conjugacy}

In~\cite[Section~5]{AF91}, a conjugacy between $F_\geo : \Omega_\geo \to \Omega_\geo$ and $F_{\bar P} : \Omega_{\bar P} \to \Omega_{\bar P}$ is provided, and in~\cite[Section~3]{AK19} similar techniques are used to construct a conjugacy between $F_\geo : \Omega_\geo \to \Omega_\geo$ and $F_\A : \Omega_\A \to \Omega_\A$ in the case where $\A$ satisfies the short cycle property. In this \namecref{sec conjugacy}, that result is extended to all extremal parameters.

First, we define bulges and corners:
\begin{defn} For $1 \le i \le 8g-4$, the \emph{lower bulge}~$\Bc_i$, the \emph{upper bulge}~$\Bc^i$, the \emph{lower corner}~$\Cc_i$, and the \emph{upper corner}~$\Cc^i$ are given by
\begin{equation*}
    \begin{split}
        \Bc_i &= (\Omega_\geo \setminus \Omega_\A) \;\cap\; \big([Q_{i+1},Q_{i+2}] \!\times\! [P_i,P_{i+1}]\big) \\
        \Bc^i &= (\Omega_\geo \setminus \Omega_\A) \;\cap\; \big([P_{i-1},P_i] \!\times\! [Q_i,Q_{i+1}]\big) \\
        \Cc_i &= (\Omega_\A \setminus \Omega_\geo) \;\cap\; \big([Q_{i+1},Q_{i+2}] \!\times\! [P_i,P_{i+1}]\big) \\
        \Cc^i &= (\Omega_\A \setminus \Omega_\geo) \;\cap\; \big([P_{i-1},P_i] \!\times\! [Q_i,Q_{i+1}]\big).
    \end{split}
\end{equation*}
Note that for a given extremal $\A$, some bulges or corners may have empty interiors.
\end{defn}

\def\PiMinusTwo{-1.42193} \def\PiMinusOne{-0.898333} \def\Pi{-0.374734} \def\PiPlusOne{0.148864} \def\PiPlusTwo{0.672463}
\def\QiMinusTwo{-1.19606} \def\QiMinusOne{-0.672463} \def\Qi{-0.148864} \def\QiPlusOne{0.374734} \def\QiPlusTwo{0.898333} \def\QiPlusThree{1.42193}
\def\BiMinusOne{-0.667399} \def\Bi{-0.0037245} \def\BiPlusOne{0.469153}
\def\CiMinusOne{-0.516315} \def\Ci{-0.108521} \def\CiPlusOne{0.353197} \def\BiPlusTwo{1.01321}
\def\GiMinusOne{-0.49} \def\Gi{0.12} \def\Hi{0.0} \def\HiPlusOne{0.6}  \def\HiPlusTwo{1.2}
\def\gap{0.523599} 
\renewcommand\temp{ \draw [out=4, in=-94] (\PiMinusTwo, \QiMinusOne) to +(\gap,\gap) to +(\gap,\gap) to [in=-130] +(0.8*\gap,0.4*\gap); \fill [white] (\PiMinusTwo, \QiMinusOne)+(0,-0.05) rectangle +(0.8*\gap,\gap); \draw [out=86, in=-176] (\Qi, \PiMinusOne) to +(\gap,\gap) to +(\gap,\gap) to +(\gap,\gap) node [above] {$\Omega_\geo$}; }
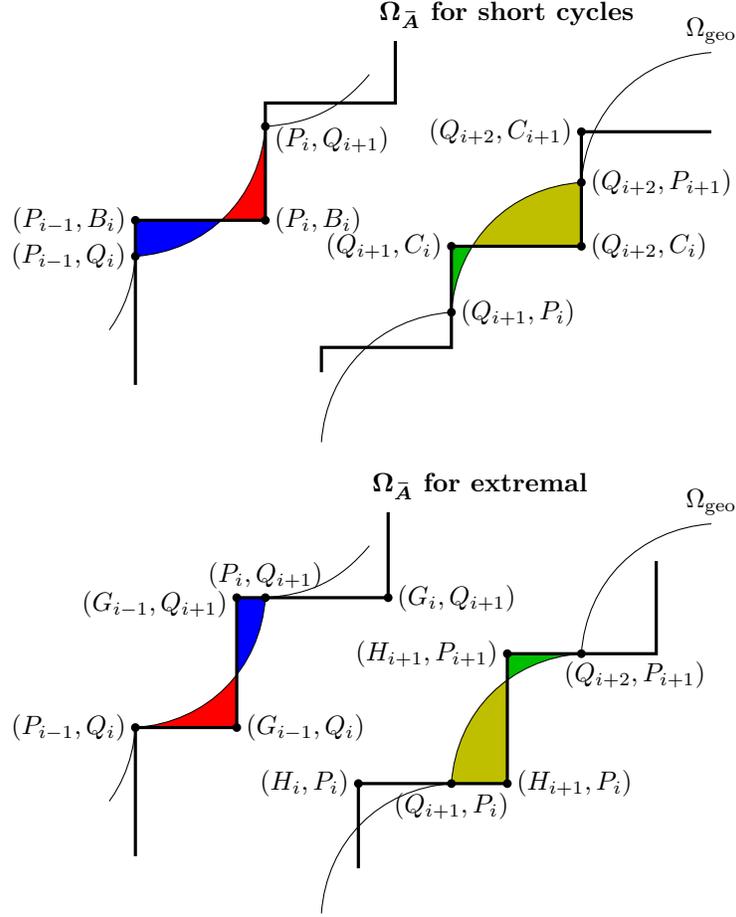
\begin{figure}[h!] 
\begin{tikzpicture}[scale=3.3]
	\begin{scope} \clip (\PiMinusOne,\Pi) rectangle (\Pi,\Bi); \fill [blue] [out=4, in=-94] (\PiMinusOne, \Qi) to (\Pi,\QiPlusOne) -- (\PiMinusOne,\QiPlusOne); \end{scope}
	\begin{scope} \clip (\PiMinusOne,\Bi) rectangle (\Pi,\QiPlusOne); \fill [red] [out=4, in=-94] (\PiMinusOne, \Qi) to (\Pi,\QiPlusOne) -- (\Pi,\Bi); \end{scope}
	\begin{scope} \clip (\QiPlusOne,\Pi) rectangle (\QiPlusTwo,\Ci); \fill [green!75!black] [out=86, in=-176] (\QiPlusOne,\Pi) to (\QiPlusTwo,\PiPlusOne) -- (\QiPlusOne,\PiPlusOne); \end{scope}
	\begin{scope} \clip (\QiPlusOne,\Ci) rectangle (\QiPlusTwo,\PiPlusOne); \fill [yellow!75!black] [out=86, in=-176] (\QiPlusOne,\Pi) to (\QiPlusTwo,\PiPlusOne) -- (\QiPlusTwo,\Pi); \end{scope}
	\temp
	\draw [very thick] (\PiMinusOne,\BiMinusOne) -- (\PiMinusOne,\Bi) -- (\Pi,\Bi) -- (\Pi,\BiPlusOne) -- (\PiPlusOne,\BiPlusOne) -- (\PiPlusOne,\BiPlusOne)--+(0,0.25) node [above=1em,right] {\hspace*{-1em}\textbf{$\boldsymbol{\Omega_\A}$ for short cycles}};
	\draw [very thick] (\Qi,\CiMinusOne-0.1) -- (\Qi,\CiMinusOne) -- (\QiPlusOne,\CiMinusOne) -- (\QiPlusOne,\Ci) -- (\QiPlusTwo,\Ci) -- (\QiPlusTwo,\CiPlusOne) -- (\QiPlusThree,\CiPlusOne);
	
	\fill (\PiMinusOne,\Qi) circle (0.5pt) node [left]  {$(P_{i-1},Q_i)$};
	\fill (\PiMinusOne,\Bi) circle (0.5pt) node [left]  {$(P_{i-1},{B}_i)$};
	\fill (\Pi,\Bi) circle (0.5pt) node [right] {$(P_i,{B}_i)$};
	\fill (\Pi,\QiPlusOne) circle (0.5pt) node [below=0.5em,right] {$(P_i,Q_{i+1})$};
	
	\fill (\QiPlusOne,\Pi) circle (0.5pt) node [right] {$(Q_{i+1},P_i)$};
	\fill (\QiPlusTwo,\Ci) circle (0.5pt) node [right] {$(Q_{i+2},{C}_i)$};
	\fill (\QiPlusOne,\Ci) circle (0.5pt) node [left]  {$(Q_{i+1},{C}_i)$};
	\fill (\QiPlusTwo,\PiPlusOne) circle (0.5pt) node [right] {$(Q_{i+2},P_{i+1})$};
	\fill (\QiPlusTwo,\CiPlusOne) circle (0.5pt) node [left]  {$(Q_{i+2},{C}_{i+1})$};

\begin{scope}[yshift=-1.9cm]
	\begin{scope} \clip (\GiMinusOne,\Qi) rectangle (\Pi,\QiPlusOne); \fill [blue] [out=4, in=-94] (\PiMinusOne, \Qi) to (\Pi,\QiPlusOne) -- (\PiMinusOne,\QiPlusOne); \end{scope}
	\begin{scope} \clip (\PiMinusOne,\Qi) rectangle (\GiMinusOne,\QiPlusOne); \fill [red] [out=4, in=-94] (\PiMinusOne, \Qi) to (\Pi,\QiPlusOne) -- (\GiMinusOne,\Qi); \end{scope}
	\begin{scope} \clip (\HiPlusOne,\Pi) rectangle (\QiPlusTwo,\PiPlusOne); \fill [green!75!black] [out=86, in=-176] (\QiPlusOne,\Pi) to (\QiPlusTwo,\PiPlusOne) -- (\QiPlusOne,\PiPlusOne); \end{scope}
	\begin{scope} \clip (\QiPlusOne,\Pi) rectangle (\HiPlusOne,\PiPlusOne); \fill [yellow!75!black] [out=86, in=-176] (\QiPlusOne,\Pi) to (\QiPlusTwo,\PiPlusOne) -- (\QiPlusTwo,\Pi); \end{scope}
	\temp
	\draw [very thick] (\PiMinusOne,\BiMinusOne) -- (\PiMinusOne,\Qi) -- (\GiMinusOne,\Qi) -- (\GiMinusOne,\QiPlusOne) -- (\Gi,\QiPlusOne) -- (\Gi,\BiPlusOne)--+(0,0.25) node [above=1em,right] {\hspace*{-1em}\textbf{$\boldsymbol{\Omega_\A}$ for extremal}};
	\draw [very thick] (\Hi,\CiMinusOne-0.2) -- (\Hi,\Pi) -- (\HiPlusOne,\Pi) -- (\HiPlusOne,\PiPlusOne) -- (\HiPlusTwo,\PiPlusOne) -- (\HiPlusTwo,\PiPlusTwo-0.15);
	
	\fill (\PiMinusOne,\Qi) circle (0.5pt) node [left]  {$(P_{i-1},Q_i)$};
	\fill (\GiMinusOne,\Qi) circle (0.5pt) node [right] {$(G_{i-1},Q_i)$};
	\fill (\GiMinusOne,\QiPlusOne) circle (0.5pt) node [below=0.25em,left] {$(G_{i-1},Q_{i+1})$};
	\fill (\Pi,\QiPlusOne) circle (0.5pt) node [above] {$(P_i,Q_{i+1})$};
	\fill (\Gi,\QiPlusOne) circle (0.5pt) node [right] {$(G_i,Q_{i+1})$};
	
	\fill (\Hi,\Pi) circle (0.5pt) node [left]  {$(H_i,P_i)$};
	\fill (\QiPlusOne,\Pi) circle (0.5pt) node [below] {$(Q_{i+1},P_i)$};
	\fill (\HiPlusOne,\Pi) circle (0.5pt) node [right] {$(H_{i+1},P_i)$};
	\fill (\HiPlusOne,\PiPlusOne) circle (0.5pt) node [left] {$(H_{i+1},P_{i+1})$};
	\fill (\QiPlusTwo,\PiPlusOne) circle (0.5pt) node [below=0.8em,right] {\hspace*{-1em}$(Q_{i+2},P_{i+1})$};
	
\end{scope}
\end{tikzpicture}

\caption{Comparison of bulges and corners for short cycle and extremal parameters. Bulges $\Bc^i$ in blue and $\Bc_i$ in gold; corners $\Cc^i$ in red and $\Cc_i$ in green.}
\label{fig bulges and corners}
\end{figure}

\begin{prop} \label{thm conjugacy bijection}
	The map $\Phi$ with domain $\Omega_\geo$ given by
	\<\label{Phi} \Phi = \left\{ \begin{array}{ll}
        \Id & \text{on } \Omega_\geo \cap \Omega_\A \\
        U_{\tau(i)+1} & \text{on } \Bc_i \\
        U_{\tau(i)} & \text{on } \Bc^i
    \end{array} \right. \>
	is a bijection from $\Omega_\geo$ to $\Omega_\A$. Specifically, $\Phi(\Bc_i) = \Cc^{\tau(i)+1}$ and $\Phi(\Bc^i) = \Cc_{\tau(i)-1}$.
\end{prop}

\begin{proof}
    All the sets $\Bc_i, \Bc^i, \Cc_i, \Cc^i$ are bounded by one horizontal line segment, one vertical line segment, and one curved segment that is part of $\partial \Omega_\geo$. Since each $U_j$ is a M\"obius transformation, we need only to show that these boundaries are mapped accordingly. The boundaries of $\Bc_i$ are the horizontal segment $[Q_{i+1},H_{i+1}] \times \{P_i\}$, part of the vertical segment $\{H_{i+1}\} \times [P_i,P_{i+1}]$, and part of the segment of $\partial \Omega_\geo$ connecting $(Q_{i+1},P_i)$ to $(Q_{i+2},P_{i+1})$. By definition, $\Phi(\Bc_i) = U_{\tau(i)+1}(\Bc_i)$. As calculated in the proof of \cite[Proposition~3.4]{AK19},
    \begin{align*}
        U_{\tau(i)+1} P_i &= Q_{\tau(i)+1}, &
        U_{\tau(i)+1} Q_{i+1} &= P_{\tau(i)}, \\
        U_{\tau(i)+1} P_{i+1} &= Q_{\tau(i)+2}, &
        U_{\tau(i)+1} Q_{i+2} &= P_{\tau(i)-1}.
    \end{align*}
    Additionally, since $H_{j} = U_{j}G_{\tau(j+1)}$ by definition and $U_{\tau(i)+1} = U_{i+1}^{-1}$, we have
    \begin{equation*}
        U_{\tau(i)+1} H_{i+1} = U_{\tau(i)+1} U_{i+1} G_{\tau(i)} = G_{\tau(i)}.
    \end{equation*}
    Therefore
    \begin{align*}
        U_{\tau(i)+1}\big( [Q_{i+1},H_{i+1}] \times \{P_i\} \big) &= [P_{\tau(i)}, G_{\tau(i)}] \times \{Q_{\tau(i)+1}\}, \\
        U_{\tau(i)+1}\big( \{H_{i+1}\} \times [P_i,P_{i+1}] \big) &= \{G_{\tau(i)}\} \times [Q_{\tau(i)+1}, Q_{\tau(i)+2}], \\
        U_{\tau(i)+1} (Q_{i+1},P_i) &= (P_{\tau(i)}, Q_{\tau(i)+1}), \\
        U_{\tau(i)+1} (Q_{i+2},P_{i+1}) &= (P_{\tau(i)-1}, Q_{\tau(i)+2}).
    \end{align*}
    The corner $\Cc^{\tau(i)+1}$ is exactly the set bounded by the horizontal segment $[P_{\tau(i)},G_{\tau(i)}] \times \{Q_{\tau(i)+1}\}$, part of the vertical segment $\{G_{\tau(i)}\} \times [Q_{\tau(i)+1}, Q_{\tau(i)+2}]$, and part of segment of $\partial \Omega_\geo$ connecting $(P_{\tau(i)}, Q_{\tau(i)+1})$ to $(P_{\tau(i)+1}, Q_{\tau(i)+2})$. Thus $U_{\tau(i)+1}(\Bc_i) = \Cc^{\tau(i)+1}$. A similar argument shows $U_{\tau(i)}(\Bc^i) = \Cc_{\tau(i)-1}$. Taking $\Phi(\Omega_\geo \cap \Omega_\A) = \Omega_\geo \cap \Omega_\A$ (since $\Phi = \Id$ on that set) together with $\Phi(\Bc_i) = \Cc^{\tau(i)+1}$ and $\Phi(\Bc^i) = \Cc_{\tau(i)-1}$, we have that $\Phi(\Omega_\geo) = \Omega_\A$.
\end{proof}

\begin{defn}
    Let $u,w\in \Sb=\partial\D$, $u\neq w$. An oriented geodesic in $\D$ from $u$ to $w$ is called \emph{$\A$-reduced} if $(u,w)\in\Omega_\A$.
\end{defn}

\Cref{thm conjugacy bijection} implies the following important facts:

\begin{samepage}
\begin{cor} \label{thm important corollaries}
	Let $\A$ be extremal, and let $\g = uw$ be a geodesic on $\D$.
	\begin{enumerate}
	\item If $\g$ intersects $\Fc$, then either $\g$ is $\A$-reduced or $U_j \g$ is $\A$-reduced, where $j=\tau(i)$ if $(u,w) \in \Bc^i$ and $j=\tau(i)+1$ if $(u,w) \in \Bc_i$.
	\item If $\g$ is $\A$-reduced, then either $\g$ intersects $\Fc$ or $\g$ intersects $U_j(\Fc)$, where $j=i$ if $(u,w) \in \Cc^i$ and $j=i+1$ if $(u,w) \in \Cc_i$.
	\end{enumerate}
\end{cor}
\end{samepage}

Using \Cref{thm important corollaries}, we can define the arithmetic cross-section.

\begin{defn} \label{defn CA}
	For $(u,w) \in \Omega_\A$, denote by $\varphi(u,w)$ the point $(z',\zeta')$ in $S\D$ where either $\g$ enters $\Fc$, or, if $\g$ does not intersect $\Fc$, the first entrance point of $\g$ to $U_j\Fc$, where $j$ is as in \Cref{thm important corollaries}(2). We define the \emph{arithmetic cross-section} $C_\A = \pi \circ \varphi(\Omega_\A)$, where $\pi$ is the canonical projection from $S\D$ to $SM$.
\end{defn}

\begin{remark} \label{alt defn CG}
	The geometric cross-section $C_\geo$ initially described in \Cref{sec intro} can be presented in an analogous way to \Cref{defn CA}. 
	For $(u,w) \in \Omega_\geo$, denote by $\psi(u,w)$ the point $(z,\zeta) \in S\D$ such that $\g$ enters $\Fc$ at $z$. Then $C_\geo$ is exactly $\pi \circ \psi(\Omega_\geo)$.
\end{remark}


\begin{prop} \label{thm conjugacy}
	The map $\Phi:\Omega_\geo \to \Omega_\A$ in \eqref{Phi} is a conjugacy between $F_\geo$ and $F_\A$, that is, $\Phi \circ F_\geo = F_\A \circ \Phi$.
\end{prop}

\begin{cor} \label{thm equal cross-sections}
	For each extremal parameter choice $\A$, the arithmetic cross-section $C_\A$ and the geometric cross-section $C_\geo$ are equal. Moreover, the first return of a tangent vector in~$C_\geo$ to~$C_\geo$ is also its first return to~$C_\A$.
\end{cor}

The proof of \Cref{thm conjugacy} is extremely similar to the proofs of~\cite[Theorem~5.2]{AF91} and \cite[Theorem~3.10]{AK19}, so we omit it here.

\begin{proof}[{Proof of \Cref{thm equal cross-sections}}]
	We use the maps $\pi\circ\varphi:\Omega_\A \to C_\A$ and $\pi\circ\psi:\Omega_\geo \to \C_\geo$ from \Cref{defn CA} and \Cref{alt defn CG}, respectively.
	Since $\pi\circ\varphi$, $\pi\circ\psi$, and $\Phi$ are bijections and $\Phi$ acts by elements of $\G$, the diagram
    \begin{equation}
        \label{cd:equalcross} \commutativeDiagram{\Omega_\geo}{\pi\circ\psi}{C_\geo}{\Phi}{\Id}{\Omega_\A}{\pi\circ\varphi}{C_\A}
    \end{equation}
    commutes.
    Indeed, let $\g=uw$ with $(u,w) \in \Omega_\geo$. If $\g$ is $\A$-reduced, then $\Phi=\Id$, $\varphi=\psi$, and the diagram commutes trivially. If not, then the geodesic $U_j \g$ is $\A$-reduced for the index $1 \le j \le 8g-4$ determined in \Cref{thm important corollaries}(1); in this case $\Phi=U_j$ and thus $\varphi\circ\Phi=\psi$, so we have $\pi \circ \varphi \circ \Phi=\pi \circ \psi$ as well.

Combining the diagram (\ref{cd:equalcross}) with \Cref{thm conjugacy}, we obtain the diagram
    \[ \begin{tikzpicture}[node distance=2cm]
        \node (top1) {$C_\geo$};
        \node [right of=top1, node distance=2.5cm] (top2) {$\Omega_\geo$};
        \node [right of=top2] (top3) {$\Omega_\geo$};
        \node [right of=top3, node distance=2.5cm] (top4) {$C_\geo$};
        \node [below of=top1] (bot1) {$C_\A$};
        \node [below of=top2] (bot2) {$\Omega_\A$};
        \node [below of=top3] (bot3) {$\Omega_\A$};
        \node [below of=top4] (bot4) {$C_\A$};
    	\path[->,font=\scriptsize,>=angle 90]
            (top1) edge node [above] {$(\pi\circ\psi)^{-1}$} (top2)
            	(top2) edge node [above] {$F_\geo$} (top3) 
    			(top3) edge node [above] {$\pi\circ\psi$} (top4)
            (top1) edge node [left]  {$\Id$} (bot1)
            (top2) edge node [left]  {$\Phi$} (bot2)
            (top3) edge node [right]  {$\Phi$} (bot3)
            (top4) edge node [right]  {$\Id$} (bot4)
            (bot1) edge node [below] {$(\pi\circ\varphi)^{-1}$} (bot2) 
            	(bot2) edge node [below] {$F_\A$} (bot3) 
    			(bot3) edge node [below] {$\pi\circ\varphi$} (bot4);
    \end{tikzpicture} \]
    The composition of the maps in the upper and lower rows are the first return maps to $C_\geo$ and $C_\A$, respectively, so the result follows from commutativity of the diagram.
\end{proof}

\begin{remark}
	The set $\Omega_\A$ is also thought to be the global attractor of $F_\A$ as a map on $\Sb\times\Sb\setminus\Delta$, $\Delta = \setbuilder{(x,x)}{x\in\Sb}$, meaning that $\bigcap_{n=0}^{\infty} F_\A^{n}(\Sb\times\Sb\setminus\Delta) = \Omega_\A$, but a proof of this is not currently known. While this attractor property is a crucial part of Zagier's ``Reduction Theory'' (see~\cite{KU10,KU17}), it not a part of Adler and Flatto's work on $\G\backslash\D$.
\end{remark}

\subsection{Symbolic coding of geodesics} \label{sec coding}

Coding of geodesics for parameters $\A$ such that $\Omega_\A$ possesses a finite rectangular structure is described in~\cite[Section~5]{AK19}. The proofs in that section do not depend explicitly on the short cycle property (indeed, it is even stated that they apply to $\A = \bar P$) but rather on~\cite[Corollaries~3.6 and~3.8]{AK19}, which \Cref{thm important corollaries} here states for extremal parameters. Thus the results of~\cite[Section~5]{AK19} do apply to all extremal parameters.

\begin{defn}
Let $\gamma = uw$ be a geodesic on $\D$ with $(u,w) \in \Omega_\A$, and denote $(u_k,w_k) = F_\A^k(u,w)$ for all $k \in \Z$.
	The \emph{arithmetic code} of $\gamma$ is the two-sided sequence
		\< \label{coding sequence} [\gamma]_\A = (\ldots,n_{-2},n_{-1},n_0,n_1,n_2,\ldots) \>
	where $n_k = \sigma(i)$ for the index $i$ such that $w_k \in [A_i,A_{i+1})$.
\end{defn}

A sequence $\overline x \in \{1,2,\ldots,8g-4\}^\Z$ is called \emph{admissible} if it can be obtained by the coding procedure above for some geodesic $uw$ with $(u,w)\in\Omega_\A$.
Given an admissible sequence $\overline x$, we can associate to it a vector $\Cod(\overline x) \in SM$ as described in~\cite[Section~5]{AK19}. The map $\Cod$ is finite-to-one and is uniformly continuous on the set of admissible coding sequences~\cite[Proposition~5.3]{AK19}, and thus we can extend it to the closure $X_\A$ of the set of all admissible sequences.
In this way we associate to each $\A$ a symbolic system $(X_\A,\sigma)$, and the geodesic flow becomes a special flow over $(X_{\A},\sigma)$ with the ceiling function $g_{\A}(\overline x)$ on $X_{\A}$ being the time of the first return to the cross-section $C_\A$ of the geodesic associated to $\overline x$.

\begin{prop} \label{thm extremal Markov}
	For any extremal parameters $\A$, the collection of all sets $[P_i,Q_i]$ and $[Q_i,P_{i+1}]$ forms a Markov partition for $f_\A$.
\end{prop}

\begin{proof}
The boundaries of Markov partition elements are measure zero, so we look at $16g-8$ total open intervals, which we denote here by $I_{2i-1} = (P_i,Q_i)$ and $I_{2i} = (Q_i,P_{i+1})$. 
Using~\cite[Proposition~2.2]{KU17}, we have
\[ \begin{split}
	f_\A(I_{2i})
	&= T_i(Q_i, P_{i+1})
	= (P_{\sigma(i)+1}, P_{\sigma(i)-2})
	\\&= I_{2\sigma(i)+1} \cup I_{2\sigma(i)+1} \cup I_{2\sigma(i)+3} \cup \cdots \cup I_{2\sigma(i)-3} \cup I_{2\sigma(i)-2}.
\end{split} \]
If $A_i = P_i$, then 
\begin{align*}
	f_\A(I_{2i-1})
	= T_i(P_i, Q_i)
	= (Q_{\sigma(i)+1}, Q_{\sigma(i)+2})
	= I_{2\sigma(i)+2} \cup I_{2\sigma(i)+3},
\end{align*}
and if $A_i = Q_i$ then 
\begin{align*}
	f_\A(I_{2i-1})
	= T_{i-1}(P_i, Q_i)
	= (P_{\tau\sigma(i)}, P_{\tau\sigma(i)+1})
	= I_{2\tau\sigma(i)-1} \cup I_{2\tau\sigma(i)}.
\end{align*}	
Thus, for every $1 \le k \le 16g-8$, we have that $f_\A(I_k)$ is a union of some sets~$I_j$.
\end{proof}

As described in~\cite[Appendix~C]{AF91}, sofic systems are obtained from Markov ones by amalgamation of the alphabet, and, conversely, Markov is obtained from sofic by refinement of the alphabet. Thus, combining each pair $I_{2k-1}$ and $I_{2k}$ into a single interval $[P_k,P_{k+1})$ gives the following:

\begin{cor} \label{thm extremal sofic}
	For any extremal parameters~$\A$, the shift on $X_\A$ is sofic with respect to the alphabet $\{1,2,\ldots,8g-4\}$.
\end{cor}

\section{Dual parameters} \label{sec dual}

The ``future'' of an arithmetic code $[uw]_\A$, that is, the terms $n_k$ with $k \ge 0$, can be determined from $w$ alone, but the ``past'' ($k<0$) generally depends on both~$u$ and~$w$. For $(a,b)$-continued fractions, the existence of ``dual codes'' (see~\cite[Section~5]{KU12}) for certain parameters allows the digits of the past to be determined from a single endpoint in those cases, and indeed a similar phenomenon can occur in the cocompact Fuchsian setting.

\begin{defn}
    Let $\A=\{A_i\}$ and $\Dual=\{D_i\}$ be two parameter choices for boundary maps from the same surface $M = \Gamma\backslash\D$ such that $F_\A$ and $F_{\Dual}$ have domains $\Omega_\A$ and $\Omega_{\Dual}$, respectively, with finite rectangular structure, and let $\phi(x, y) = (y, x)$.
    We say that $\Dual$ is \emph{dual} to $\A$ if $\phi(\Omega_\A) = \Omega_\Dual$ and $\phi(F_\A^{-1}(p)) = F_\Dual(\phi(p))$ for all $p = (u,w) \in \Omega_\A$ with $u \notin \Dual$.\footnote{\,The definition of duality in~\cite[Definition~9.1]{AK19} states $\phi(F_\A^{-1}(p)) = F_{\A'}(\phi(p))$ for all $p \in \Omega_\A$, but no examples of duality are proved in that paper. Here, the definition's requirement is for $(u,w) \in \Omega_\A$ with $u \notin \A'$.}
\end{defn}

\begin{samepage}
\begin{thm}[{\cite[Theorem~9.2]{AK19}}]
	If $\Dual$ is dual to $\A$ and $(u,w) \in \Omega_\A$, then the arithmetic code $[\gamma]_\A$ of the geodesic $\gamma = uw$ satisfies 
	\begin{itemize}
		\item for $k \ge 0$, $n_k$ is the value of $i$ for which $f_\A^k(w) \in [A_{\sigma(i)},A_{\sigma(i)+1})$, and
		\item for $k < 0$, $n_k$ is the value of $i$ for which $f_{\Dual}^{-k+1}(u) \in [D_i,D_{i+1})$.
	\end{itemize}
\end{thm}
\end{samepage}

Dual codes do not exist within the class of parameters with short cycles~\cite[Proposition~9.3]{AK19}. Prior to 2018, the only known examples of dual parameters were~$\bar P$ and~$\bar Q$, both of which are extremal, and in fact a search for additional examples of duality was a primary motivation for studying the class of extremal parameters. \Cref{thm dual} shows that every extremal $\A$ has a corresponding dual parameter set.

\begin{thm} \label{thm dual}
	Given any extremal parameters $\A$, the parameter choice $\Dual = \{D_1,$ \ldots, $D_{8g-4}\}$ from \Cref{thm system solution} and Equation~\ref{defn H D} is dual to $\A$.
\end{thm}

Note that $\Dual$ is not necessarily extremal, nor will it satisfy the short cycle property. Since $\Dual$ may not satisfy the conditions of \Cref{thm extremal bijectivity} or~\cite[Theorem~1.3]{KU17}, the domain for $F_\Dual$ must be described and proven independently.

\begin{thm} \label{thm dual bijectivity}
	Given any extremal parameters $\A$, let $\{G_i\}$, $\{H_i\}$, and $\{D_i\}$ be given by \Cref{thm system solution} and Equation~\ref{defn H D}. Then then set
	\< \label{Omega dual} \Omega_\Dual := \bigcup_{i=1}^{8g-4} \raisebox{-0.5em}{$\begin{array}{ll} [Q_{i+2},P_{i-1}] \!\times\! [D_i,D_{i+1}] \\ \quad\cup\; [P_{i-1},P_i] \!\times\! [H_i,D_{i+1}] \;\cup\; [Q_{i+1},Q_{i+2}] \!\times\! [D_i,G_i] \end{array}$} \>
	is a bijectivity domain for the map $F_\Dual$.
\end{thm}

\begin{proof}
	Define the following rectangles:
	\begin{align*}
    	R_i &:= [Q_{i+2},P_{i-1}] \times [D_i,D_{i+1}) \\*
		R_i^u &:= [P_{i-1},P_i] \times [H_i,D_{i+1}) \\*
		R_i^\ell &:= [Q_{i+1},Q_{i+2}] \times [D_i,G_i].
	\end{align*}
	These rectangles may have empty interior (see Cases~\ref{case-upper-empty} and~\ref{case-lower-empty} below), but we include them in calculations for now.
	
	Although initially $\Omega_\Dual$ was defined as $\bigcup_{i=1}^{8g-4} R_i \cup R_i^u \cup R_i^\ell$, it can equivalently be described via a decomposition into vertical strips as $\Omega_\Dual = \bigcup_{i=1}^{8g-4} V_i \cup V_i'$, where
	\< \begin{split} V_i &:= [P_i,Q_i] \times [H_{i+1},G_{i-2}) \\ V_i' &:= [Q_i,P_{i+1}] \times [H_{i+1},G_{i-1}). \end{split} \label{Omega dual vertical strips} \>
	
	The map $F_\Dual$ acts on the rectangles $R_i, R_i^u, R_i^\ell$ by $T_i$.
	From \Cref{lem DGH properties}(1), we have $T_i D_i = H_{\sigma(i)+1}$ for all $i$, and from the definition of $D_{i+1}$ we have $T_i D_{i+1} = G_{\sigma(i)-1}$ for all $i$. Setting $j = \sigma(i)$ for the remainder of this proof, we have
	\begin{align*}
    	F_\Dual R_i 
		&= [T_iQ_{i+2},T_iP_{i-1}] \times [T_iD_i,T_iD_{i+1}) \\*
		&= [Q_{j}, P_{j+1}] \times [H_{j+1}, G_{j-1}) \\*
		&= V_j'.
	\end{align*}
	
	We now consider the following four cases:
	\begin{enumerate}[\quad\text{Case }I:]
		\item \label{case-upper} $A_{j} = P_{j}$.
		\item \label{case-upper-empty} $A_{j} = Q_{j}$.
		\item \label{case-lower} $A_{j+1} = Q_{j+1}$.
		\item \label{case-lower-empty} $A_{j+1} = P_{j+1}$.
	\end{enumerate}
	Note that Cases~\ref{case-upper} and~\ref{case-upper-empty} are mutually exclusive, as are Cases~\ref{case-lower} and~\ref{case-lower-empty}.
	
	\medskip
	\textbf{Case~\ref{case-upper}.}
	When $A_{j} = P_{j}$,~\eqref{system} gives \[ G_i = G_{\sigma(j)} = T_{j} G_{j-2}, \] and applying $T_i = T_j^{-1}$ to the left and right sides gives $T_iG_i = G_{j-2}$. 
	From \Cref{lem DGH properties}(1), we have $T_i D_i = H_{\sigma(i)+1}$ for all $i$. Thus
	\begin{align*}
    	F_\Dual R_i^\ell 
		&= [T_iQ_{i+1},T_iQ_{i+2}] \times [T_iD_i,T_iG_i] \\*
		&= [P_{j}, Q_{j}] \times [H_{j+1}, G_{j-2}] \\*
		&= V_j. 
	\end{align*}
	
	\medskip
	\textbf{Case~\ref{case-upper-empty}.}
	When $A_{j} = Q_{j}$,~\eqref{system} gives 
	\[ G_i = G_{\sigma(j)} = T_{\tau(j)+1} G_{\tau(j)} =: D_i. \]
	Having $G_i = D_i$ means that the rectangle $R_i^\ell$ has empty interior. One could consider $R_i^\ell$ to be the horizontal segment $[Q_{i+1},Q_{i+2}] \times \{D_i\}$, but this segment is the upper-left part of the boundary of $\overline{R_{i+1}}$, so it already handled by the considerations of $\{R_k\}_{k=1}^{8g-4}$.
	
	\medskip
	\textbf{Case~\ref{case-lower}.}
	When $A_{j+1} = Q_{j+1}$,~\eqref{system} gives $T_{\tau(j)+2} G_{\tau(j)+1} = G_{\tau\sigma(j)-1} = G_{\tau(i)-1}$, so
	\begin{align*}
		T_i H_i
		&= T_i U_i G_{\tau(i)-1} \\
		&= T_i (T_{j} T_{\tau(i)-1}) (T_{\tau(j)+2} G_{\tau(j)+1}) \\
		&= T_{\tau(i)-1} T_{\tau(j)+2} G_{\tau(j)+1} \\
		&= U_{j+2} G_{\tau(j)+1} \\*
		&= H_{j+2}.
	\end{align*}
	We also have $T_i D_{i+1} = G_{\sigma(i)-1}$ from the definition of $D_{i+1}$. Therefore
	\begin{align*}
    	F_\Dual R_i^u 
		&= [T_iP_{i-1},T_iP_i] \times [T_iH_i, T_iD_{i+1}) \\
		&= [P_{j+1}, Q_{j+1}] \times [H_{j+2}, G_{j-1}) \\*
		&= V_{j+1}.
	\end{align*}
	
	\medskip
	\textbf{Case~\ref{case-lower-empty}.}
	When $A_{j+1} = P_{j+1}$,~\eqref{system} gives $G_{\tau(i)-1} = G_{\sigma(j+1)} = T_{j+1} G_{j-1}$, and applying $T_{j+1}^{-1} = T_{\tau(i)-1}$ to the left and right sides gives
	\begin{align*}
    	T_{\tau(i)-1} G_{\tau(i)-1} &= G_{j-1} \\
		T_{\sigma(i)} T_{\tau(i)-1} G_{\tau(i)-1} &= T_{j} G_{j-1} \\
		U_i G_{\tau(i)-1} &= T_{j} G_{j-1} \\
		H_i &= D_{\sigma(j)+1} = D_{i+1}.
	\end{align*}
	Having $H_i = D_{i+1}$ means that the rectangle $R_i^u$ has empty interior. One could consider $R_i^u$ to be the horizontal segment $[P_{i-1},P_i] \times \{D_{i+1}\}$, but this segment is the lower-right part of the boundary of $\overline{R_{i+1}}$, so it already handled by the considerations of $\{R_k\}_{k=1}^{8g-4}$.
	
\medskip
Cases~\ref{case-upper-empty} and~\ref{case-lower-empty} show that
\[
    \Omega_\Dual
    = \bigcup_{i=1}^{8g-4} R_i \cup \bigcup_{\substack{i \\ A_{\sigma(i)+1}=Q_{\sigma(i)+1}}} R_i^u \cup \bigcup_{\substack{i \\ A_{\sigma(i)}=P_{\sigma(i)}}} R_i^\ell 
\]
and Cases~\ref{case-upper} and~\ref{case-lower}, along with the initial explanation that $T_iR_i = V_{\sigma(i)}'$ for all $i$, show that
\begin{align*}
    F_\Dual(\Omega_\Dual)
    &= \bigcup_{i=1}^{8g-4} T_iR_i \cup \bigcup_{\substack{i \\ A_{j+1}=Q_{j+1}}} T_iR_i^u \cup \bigcup_{\substack{i \\ A_{j}=P_{j}}} T_iR_i^\ell, \qquad j=\sigma(i), \\
    &= \bigcup_{j=1}^{8g-4} V_j' \cup \bigcup_{\substack{j \\ A_{j+1}=Q_{j+1}}} V_{j+1} \cup \bigcup_{\substack{j \\ A_{j}=P_{j}}} V_j \\
    &= \bigcup_{k=1}^{8g-4} (V_k \cup V_k') \\*
    &= \Omega_\Dual.
\end{align*}
Every $T_i$ is a bijection, so $F_\Dual$ is a bijection on the interior of every $R_i, R_i^u, R_i^\ell$.
\end{proof}

Now that $\Omega_\Dual$ from~\eqref{Omega dual} has been established as the domain of $F_\Dual$, we can prove that $\Dual$ is dual to $\A$.

\begin{proof}[{Proof of \Cref{thm dual}}]
Recall $\phi(x,y) = (y,x)$. Comparing~\eqref{Omega E horizontal strips} and~\eqref{Omega dual vertical strips}, we see that 
\begin{align*}
    \phi(R_i'') &= [P_i,Q_i] \times [H_{i+1},G_{i-2}] = V_i, \\*
    \phi(R_i') &= [Q_i,P_{i+1}] \times [H_{i+1},G_{i-1}] = V_i',
\end{align*}
and since 
\[
	\Omega_\A = \bigcup_{i=1}^{8g-4} (R_i'' \cup R_i')
	\qquad\text{and}\qquad
	\Omega_\Dual = \bigcup_{i=1}^{8g-4} (V_i \cup V_i'),
\]
we have exactly $\phi(\Omega_\A) = \Omega_\Dual$.

\medskip
Let $(x_0,y_0) \in \Omega_\A$, and let $(x_1,y_1) := F_\A(x_0,y_0)$. 
Let $i$ be such that $y_0 \in [A_i,A_{i+1})$, and thus $(x_1,y_1) = (T_ix_0, T_iy_0)$.
Since $A_i$ might be either $P_i$ or $Q_i$ and $A_{i+1}$ might be either $P_{i+1}$ or $Q_{i+1}$, we must consider that $y_0$ could be in $[P_i, Q_i]$ or $[Q_i, P_{i+1}]$ or $[P_{i+1}, Q_{i+1}]$.

\begin{itemize}
	\item If $y_0 \in [P_i, Q_i)$ then from the decomposition of $\Omega_\A$ into horizontal strips in~\eqref{Omega E defn}, we must have that $x_0 \in [H_{i+1},G_{i-2}]$. Thus
\[ x_1 = T_i x_0 \in [T_iH_{i+1}, T_iG_{i-2}] = [D_{\sigma(i)}, T_iG_{i-2}]. \]
Since $G_{i-2} \in [P_{i-2},P_{i-1}]$, its image $T_i G_{i-2}$ is in
\[ [T_iP_{i-2}, T_iP_{i-1}] = [T_iP_{i-2}, P_{\sigma(i)+1}] \subset [Q_{\sigma(i)}, P_{\sigma(i)+1}], \]
and thus $x_1 \in [D_{\sigma(i)}, P_{\sigma(i)+1}] \subset [D_{\sigma(i)}, D_{\sigma(i)+1}]$.

	\item If $y_0 \in [Q_i, P_{i+1})$ then from~\eqref{Omega E defn}, we have $x_0 \in [H_{i+1},G_{i-1}]$, so
\[ x_1 = T_i x_0 \in [T_iH_{i+1}, T_iG_{i-1}] = [D_{\sigma(i)}, D_{\sigma(i)+1}]. \]

	\item If $y_0 \in [P_{i+1}, Q_{i+1})$ then from~\eqref{Omega E defn}, we have $x_0 \in [H_{i+2},G_{i-1}]$, so
\[ x_1 = T_i x_0 \in [T_iH_{i+2}, T_iG_{i-1}] = [T_iH_{i+2}, D_{\sigma(i)+1}]. \]
Since $H_{i+2} \in [Q_{i+2},Q_{i+3}]$ by \Cref{lem DGH properties}(3), its image $T_iH_{i+2}$ is in
\[ [T_iQ_{i+2},T_iQ_{i+3}] = [Q_{\sigma(i)},T_iQ_{i+3}] \subset [Q_{\sigma(i)},Q_{\sigma(i)+1}], \]
and thus $x_1 \in [Q_{\sigma(i)}, D_{\sigma(i)+1}] \subset [D_{\sigma(i)}, D_{\sigma(i)+1}]$.
\end{itemize}

In all three cases, $x_1 \in [D_{\sigma(i)}, D_{\sigma(i)+1}]$, and since we restrict to $x_1 \notin \Dual$, this means $F_\Dual$ acts on $(y_1,x_1)$ by $T_{\sigma(i)} = T_i^{-1}$ and therefore
\[ F_\Dual(y_1,x_1) = (T_i^{-1}y_1,T_i^{-1}x_1) = (T_i^{-1}T_iy_0, T_i^{-1}T_ix_0) = (y_0,x_0) = \phi(x_0,y_0). \]
Thus with $p = (x_1,y_1)$ we have that $\phi(F_\A^{-1}(p)) = F_\Dual(\phi(p))$ so long as $x_1 \notin \Dual$.
\end{proof}

\subsection{Examples of duality} \label{sec examples}

In this \namecref{sec examples}, we first compute the dual to the specific genus-$2$ example parameters from~\eqref{example A} and then give some descriptions of dual parameter choices that exist for any genus. 

\medskip Recall the example parameter choice
\[ \A = \{ P_1, P_2, P_3, P_4, Q_5, P_6, Q_7, Q_8, P_9, P_{10}, Q_{11}, Q_{12} \} \]
in~\eqref{example A}. 
Using~\eqref{defn H D} and the values $G_1,\ldots,G_{12}$ for this example given in~\eqref{example G}, we directly compute
\begin{align*}
	D_1 &= T_{2}P_{1} = P_1   &  D_7 &= T_{8}P_8 = Q_7 \\*
	D_2 &= T_{7}P_6 = P_2     &  D_8 &= T_{1}P_1 = Q_8 \\*
	D_3 &= T_{12}P_{12} = Q_3 &  D_9 &= T_{6}T_3P_3 = T_{11}P_{1} \\*
	D_4 &= T_{5}P_5 = Q_4     &  D_{10} &= T_{11}T_4P_2 = T_{4}P_{6} \\*
	D_5 & = T_{10}T_6T_3P_1 = T_{10}T_{11}P_1  &  D_{11} &= T_{4}Q_{3} \\* 
	D_6 &= T_{3}P_2 = P_6     &  D_{12} &= T_{9}P_9 = Q_{12}.
\end{align*}
Therefore, the parameter choice
\< \label{example D} \Dual = \{ P_1,\:  P_2,\:  Q_3,\:  Q_4,\:  T_{10}T_{11}P_1,\:  P_6,\:  Q_7,\:  Q_8,\:  T_{11}P_1,\:  T_4P_6,\:  T_4Q_3,\: Q_{12} \} \>
is dual to
\[ \A = \{ P_1, P_2, P_3, P_4, Q_5, P_6, Q_7, Q_8, P_9, P_{10}, Q_{11}, Q_{12} \}. \]
The domains $\Omega_\A$ and $\Omega_\Dual$ for this example are shown in Figures~\ref{fig attractor example} and~\Cref{fig dual attractor example}, respectively. Comparing these two figures, one can see that $\phi(\Omega_\A) = \Omega_\Dual$, where $\phi(x,y) = (y,x)$.

\begin{figure}[ht]
	\includegraphics[width=0.67\textwidth]{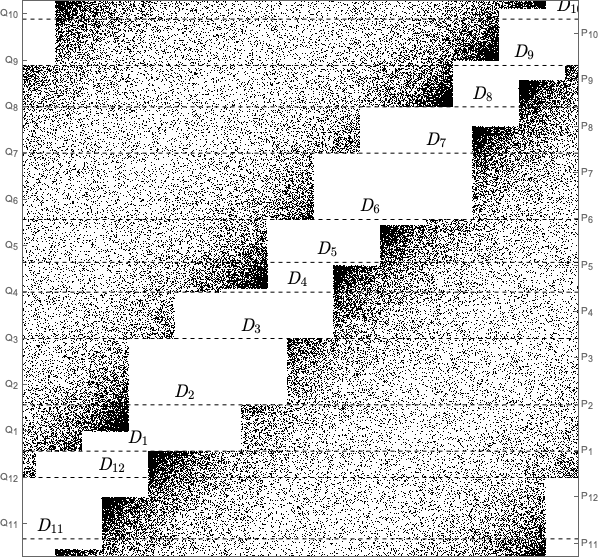}
	\caption{The domain $\Omega_\Dual$ for the parameters $\Dual$ in~\eqref{example D}.}
	\label{fig dual attractor example}
\end{figure}

\medskip
The following \namecref{thm dual examples} gives six types of extremal parameter choices whose duals can be described easily for any genus.

\begin{prop} \label{thm dual examples}
	For any $g \ge 2$,
	\begin{enumerate}[\quad a)]
		\item $\P$ and $\bar Q$ are dual to each other.
		\item \label{alternating dual} $\{P_1,Q_2,P_3,Q_4,\ldots\}$ and $\{Q_1,P_2,Q_3,P_4,\ldots\}$ are dual to each other.
		\item $\{P_1,P_2,Q_2,Q_2,P_3,P_4,Q_5,Q_6,\ldots\}$ is dual to itself.
		\item $\{Q_1,Q_2,P_2,P_2,Q_3,Q_4,P_5,P_6,\ldots\}$ is dual to itself.
	\end{enumerate}
\end{prop}

\noindent The duality of the classical cases $\bar P$ and $\bar Q$ is shown in \cite{AF91}. The domains for the other four parameters choices are shown for genus $g=2$ in \cite[Figure~12]{AK19}.

\begin{proof}[{Proof of \Cref{thm dual examples}(\ref{alternating dual})}]
	Let $A_i=P_i$ for odd $i$ and $A_i = Q_i$ for even $i$. When $i$ is odd,~\eqref{system} and the definition of $\sigma$ from~\eqref{sigma} give that $G_{4g-i} = T_iG_{i-2}$. Setting $j = \sigma(i-2) = 4g-(i-2) = 4g+2-i$, which is also odd, we also get that $G_{\sigma(j)} = T_jG_{j-2}$, or, equivalently, $G_{i-2} = T_{4g+2-i}G_{4g-i}$. Therefore
	\[ G_{i-2} = T_{4g+2-i}G_{4g-i} = T_{4g+2-i}(T_iG_{i-2}) = T_{\sigma(i-2)}T_iG_{i-2}. \]
	The fixed points of $T_{\sigma(i-2)}T_i$ are $P_{i-1}$ and $Q_i$, and since $G_{i-2} \in [P_{i-2},P_{i-1}]$, we have that $G_{i-2} = P_{i-1}$. Since this holds for all even $i$, we can reindex to say that
	\[ G_i = P_{i+1} \qquad\text{for $i$ odd.} \]
	For $k$ even,~\eqref{system} and~\eqref{sigma} and~\eqref{tau} give that $G_{2-k} = T_{k+4g-1}G_{k+4g-2}$ and that $G_{k+4g-2} = T_{3-k}G_{2-k}$, so
	\[ G_{2-k} = T_{k+4g-1}T_{3-k}G_{2-k} = T_{\tau(k)+1}T_{\sigma(k)+1}G_{2-k}, \]
	and since the fixed points of $T_{\tau(k)+1}T_{\sigma(k)+1}$ are $P_{\sigma(k)}$ and $Q_{\sigma(k)+1}$, we must have $G_{2-k} = G_{\sigma(k)} = P_{\sigma(k)} = P_{2-k}$. Reindexing to $i = 2-k$, we have that
	\[ G_i = P_{i} \qquad\text{for $i$ even.} \]
	Having determined $G_1,\ldots,G_{8g-4}$, we use~\eqref{defn H D} to compute that
	\[ D_i = T_{\tau\sigma(i)+1}G_{\tau\sigma(i)} = \left\{\begin{array}{ll}
		T_{\tau\sigma(i)+1}P_{\tau\sigma(i)+1} &\text{if $i$ odd} \\
		T_{\tau\sigma(i)+1}P_{\tau\sigma(i)} &\text{if $i$ even}.
	\end{array}\right. \]
	In general $T_{\tau\sigma(i)+1}P_{\tau\sigma(i)+1} = Q_i$ and $T_{\tau\sigma(i)+1}P_{\tau\sigma(i)} = P_i$, so this gives exactly $\Dual = \{Q_1,P_2,Q_3,P_4,Q_5,\ldots\}$ as the dual to $\A = \{P_1,Q_2,P_3,Q_4,\ldots\}$. 
\end{proof}


\begin{thebibliography}{99}

\bibitem{AK19} A.~Abrams, S.~Katok. {Adler and Flatto revisited: cross-sections for geodesic~flow on compact surfaces of constant negative curvature}, \textit{Studia Mathematica} \textbf{246} (2019), 167--202. 

\bibitem{AKU20} A.~Abrams, S.~Katok, I.~Ugarcovici. {Flexibility of entropy of boundary maps for surfaces of constant negative curvature.} Submitted to \textit{Ergodic Theory and Dynamical Systems}, September 2019. \texttt{\color{blue}arxiv.org/abs/1909.07032}

\bibitem{A98} R.~Adler. {Symbolic dynamics and Markov partitions}, \textit{Bull. Amer. Math. Soc.} \textbf{35} (1998), No. 1, 1--56.

\bibitem{AF91} R.~Adler, L.~Flatto. {Geodesic flows, interval maps, and symbolic dynamics}, \textit{Bull. Amer. Math. Soc.} \textbf{25} (1991), No.~2, 229--334.

\bibitem{BiS87} J.~Birman, C.~Series. {Dehn's algorithm revisited, with applications to simple curves on surfaces}, \textit{Combinatorial Group Theory and Topology} (AM-111), Princeton University Press, (1987), 451--478.

\bibitem{B88} F.~Bonahon. {The geometry of Teichm\"uller space via geodesic currents}, \textit{Inventiones Mathematicae} \textbf{92} (1988), 139--162. 

\bibitem{BS79} R.~Bowen, C.~Series. {Markov maps associated with Fuchsian groups}, \textit{Inst. Hautes \'Etudes Sci. Publ. Math.} \textbf{50} (1979), 153--170.

\bibitem{KH} A.~Katok, B.~Hasselblatt. \textit{Intro.~to the Modern Theory of Dynamical Systems}, Cambridge University Press, 1995.

\bibitem{K96} S.~Katok. {Coding of closed geodesics after Gauss and Morse.} \textit{Geom. Dedicata} \textbf{63} (1996), 123--145.

\bibitem{KU10} S.~Katok, I.~Ugarcovici. {Structure of attractors for $(a,b)$-continued fraction transformations.} \textit{Journal of Modern Dynamics} \textbf{4} (2010), 637--691.

\bibitem{KU12} S.~Katok, I.~Ugarcovici. {Applications of $(a,b)$-continued fraction transformations.} \textit{Ergodic Theory and Dynamical Systems} \textbf{32} (2012), 755--777. 

\bibitem{KU17} S.~Katok, I.~Ugarcovici. {Structure of attractors for boundary maps associated to Fuchsian groups}, \textit{Geometriae Dedicata} \textbf{191} (2017), 171--198.

\bibitem{KU17e} S.~Katok, I.~Ugarcovici. {Errata: Structure of attractors for boundary maps associated to Fuchsian groups}, \textit{Geometriae Dedicata} \textbf{198} (2019), 189--181.

\bibitem{Ko29} P.~Koebe. {Riemannsche Mannigfaltigkeiten und nicht euklidische Raumformen}, IV, \textit{Sitzungsberichte Deutsche Akademie von Wissenschaften}, (1929), 414--557.

\bibitem{W92} B.~Weiss. {On the work of Roy Adler in ergodic theory and dynamical systems}, in \textit{Symbolic dynamics and its applications} (New Haven, CT, 1991), 19--32, \textit{Contemp. Math.} \textbf{135}, Amer. Math. Soc., Providence, RI, 1992.

\end{thebibliography}
\end{document}